\documentclass{article}

\usepackage{cite}
\usepackage{amsfonts}
\usepackage{amsmath}
\usepackage{amsthm}
\usepackage{appendix}

\usepackage{enumerate}
\usepackage[shortlabels]{enumitem}

\usepackage{ulem}

\usepackage{color}
\usepackage[a4paper]{geometry}
\def\o{\otimes}

\def\ran{\text{ran}}
\def\s{\boldsymbol{s}}
\def\t{\boldsymbol{t}}

\theoremstyle{definition}
\newtheorem{definition}{Definition}[section]
\newtheorem{notation}{Notation}
\newtheorem{remark}[definition]{Remark}
\newtheorem{example}[definition]{Example}

\theoremstyle{plain}
\newtheorem{theorem}[definition]{Theorem}
\newtheorem{proposition}[definition]{Proposition}
\newtheorem{lemma}[definition]{Lemma}
\newtheorem{corollary}[definition]{Corollary}

\numberwithin{equation}{section}

\title{Coactions of weak multiplier Hopf algebras}
\author{Nan Zhou \thanks{Email: nanzhou.math@zjsru.edu.cn, address: Zhejiang Shuren University, 310015, Hangzhou, China.}}
\date{\today}
     
\begin{document}

\maketitle

\begin{abstract}
In this article, we investigate the notion of a coaction for a weak multiplier Hopf algebra. We show that there is a duality for actions and coactions of weak algebraic quantum groups. Then we obtain the well known Blattner-Montgomery duality theorem for weak algebraic quantum groups in the framework of dual pairs. We form a Morita context when the weak multiplier Hopf algebra admits a two-sided faithful integral and show that the  coaction is Galois if and only if the map in the Morita context is surjective. As an application, we obtain the result that a grading by any groupoid $G$ can be considered as a $\mathrm{C}G$-comodule algebra. We also construct a Morita context connecting the smash product $A\# \mathrm{C}(G)$ and the coinvariant subalgebra $A^{\mathrm{C}(G)}$. Finally we construct an isomorphism between  $(R\# \mathrm{C}(G))\# \mathrm{C}G$ and $R\o (\mathrm{C}(G)\# \mathrm{C}G)$ when $R$ is graded by $G$.
\end{abstract}

\section{Introduction}
	
Let $(A, \Delta, \varepsilon)$ be a coalgebra. A right coaction of $A$ over a vector space $V$ is a linear map $\Gamma: V\rightarrow V\o A$ such that 
\begin{equation}\label{coass-usual}
(\Gamma \o id)\Gamma=(id\o \Delta)\Gamma
\end{equation}
and
$
(id\o \varepsilon)\Gamma=id
$. We use $id$ for the identity map. When $V$ is an algebra, then $V$ is called a left $A$-comodule algebra if the coaction map $\Gamma$ is a unital homomorphism. The standard and motivating example is as follows.

Let $G$ be a finite group and let $A$ be the space of complex functions on $G$. The coproduct $\Delta$ and counit $\varepsilon$ 
on $A$ are given as $\Delta(f)(p,q)=f(pq), \varepsilon(f)=f(e)$, for any $f\in A$ and $ p,q\in G$, here $e$ is the unit element in $G$.  Assume now that $(p, r)\mapsto \alpha_p(r)$ is a left action of $A$ on a vector space $R$. We may identify the tensor product $R\o A$ with the space of functions on $G$ with values in $R$ by associating $r\o f$ with the function $p\mapsto f(p)r$. Now, the map $\Gamma: R\rightarrow R\o A$ defined as $\Gamma(r)p=\alpha_p(r)$ is a right coaction $A$ on $R$. It is easy to check that $\Gamma$ satisfies the coassociativity equation \eqref{coass-usual}.

If the group $G$ is no longer finite and if $R$ is not finite-dimensional, the above construction does not work anymore because $\Gamma(f)$ will not belong to $A\o R$. Now we may require $A$ to be a multiplier Hopf algebra and not a Hopf algebra. Then the coaction (comodule algebra) of a multiplier Hopf algebra $A$ on a non-degenerate algebra $R$ is an injective homomorphism $\Gamma: R\rightarrow M(R\o A)$ such that $\Gamma(R)(1\o A)$ and $(1\o A)\Gamma(R)$ are contained in  $R\o A$  and
\begin{equation}\label{coass-1}
(\Gamma\o id)(\Gamma(r)(1\o b))=(id\o \Delta)\Gamma(r)(1\o 1\o b).
\end{equation}
The counitary property will follow from the injectivity of $\Gamma$. The notion was first given by Van Daele and Zhang in Section 2 of \cite{van1999galois}. Note that only in the case of regular multiplier Hopf algebras and when $R$ is idempotent, the coassociativity equation can be defined as equation \eqref{coass-usual} because $\Gamma$ is nondegenerate by \cite[Proposition 2.5]{van1999galois}. In fact, we will show that equation \eqref{coass-usual} can be meaningful for any weak multiplier Hopf algebra $A$ and non-idempotent algebra $R$, see Corollary \ref{general-coass}. The basic example is modified as follows.

Let $G$ be any group and we take $\mathrm{C}(G)$ as the algebra of complex functions with finite support. We denote this algebra by $A$. The coproduct $\Delta$ is defined as before. Assume that $R$ is an algebra and that $\alpha_p$ is an automorphism of $R$ for every $p\in G$. We can view the evaluation at a point $p$ of $G$ as a homomorphism of $\mathrm{C}(G)$ to $\mathbb{C}$. This will yield an evaluation map from $M(R\o A)$ to $M(R)$. From this point of view, with the same definition as above, $\Gamma(r)(p)$ is the image under this evaluation map of the element $\Gamma(r)\in M(R\o A)$. We may denote
\begin{equation}\label{gamma}
\Gamma(r)=\sum_{p\in G}\alpha_p(r)\o \delta_p
\end{equation}
where $\delta_p$ is the function that  is $1$  in $p$ and $0$ everywhere else. The sum is defined in the strict topology of the multiplier algebra in the sense that we get a finite sum if we multiply with an element of $R\o A$. It is easy to check that $\Gamma$ is a coaction. 

Another related object is the groupoid. We will consider the coaction of a groupoid algebra. Let $G$ be a {\itshape finite} groupoid and $\mathrm{C}(G)$ the algebra of complex functions on $G$. Now the coproduct on $\mathrm{C}(G)$ is defined as 
$$\Delta(f)(p, q)=\begin{cases}
f(pq),\quad &\text{if $pq$ is defined,}\\
0, \quad &\text{otherwise}.
\end{cases}$$
Then $\mathrm{C}(G)$ is a weak Hopf algebra in the sense of \cite{bohm1999weak}. Let $\alpha$ be a {\itshape true} action of $G$ {\itshape on} a set $X$ (see the definition in Section 2.1 or \cite[Example 3.2.2]{zhou2017duality}).  Denote by $X_p$  the set of elements on which $p$ acts. Let $R$ be the algebra $\mathrm{C}(X)$ of complex functions with finite support on X and pointwise product. For any $p$, the map $\alpha_p$ from $R$ to $R$ is defined as
 $$\alpha_p(r)(x)=\begin{cases}
r(\alpha_{p^{-1}}(x)), \quad &x\in X_{p^{-1}},\\
0, & \text{otherwise.}
 \end{cases}$$
At this time, the coaction map can still be given by  the equation \eqref{gamma}. Note that we will not have $\Gamma(1)=1\o 1$,  because not all $x$ belong to $X_p$. And if $X_p=X$ for all $p\in G$, then $G$ must be a group by the trueness of the action.

Things become more complicated when $G$ is a groupoid with an infinite set of units (and $R$ is not unital). 
There is a problem with coassociativity as formulated above because the map $\Gamma\o id$ is defined on $R\o A$, but not necessarily on the range of the coaction map in $M(R\o A)$. Therefore, the left-hand side of equation \eqref{coass-usual} does not make sense. Fortunately, the appearance of weak multiplier Hopf algebras provides an approach to the solution of this problem. Obviously, if $A$ is a unital algebra (such as Hopf algebras, weak Hopf algebras), we do not need to worry about the coassociativity. In fact, this is a crucial problem when we consider coassociativity on the non-unital algebras, such as the coproduct of multiplier Hopf algebras, the coaction of multiplier Hopf algebras and the coproduct of weak multiplier Hopf algebras. For a general consideration of the problem, we refer to \cite{van2024reflections1}.

In Section 3 we will give a notion of coactions of weak multiplier Hopf algebras such that also the above example, with an infinite groupoid, fits into the scheme. We will follow the philosophy used to define weak multiplier Hopf algebras in \cite{van2015weak} by Van Daele and Wang.  

Let $A$ be a weak multiplier Hopf algebra and $R$ a nondegenerate algebra. Let $\Gamma$ be a homomorphism from $R$ to $M(R\o A)$.  The first thing is that we need to formulate a suitable coassociativity of $\Gamma$. We start with two maps $W_1$ and $W_3$ which are defined from $R\o A$ to $M(R\o A)$ as
$$W_1(r\o a)=\Gamma(r)(1\o a)\quad \text{and} \quad W_3(r\o a)=(1\o a)\Gamma(r).$$
It is assumed that the ranges of $W_1$ and $W_3$ are given by an idempotent $P\in M(R\o A)$ in the sense that 
$$\Gamma(R)(1\o A)=P(R\o A) \quad \text{and} \quad (1\o A)\Gamma(R)=(R\o A)P.$$
Under this assumption we can extend $\Gamma$ to a homomorphism from $M(R)$ to $M(R\o A)$ such that $\Gamma(1)=P$ ($\Gamma(1)$ may differ from $1\o 1$). Similarly, the homomorphism $\Gamma\o id$ has a unique extension to $M(R\o A)$. Now, the coassociativity equation can be expressed as 
\begin{equation}\label{coass-R}
(\Gamma\o id)\Gamma(r)=(id\o \Delta)\Gamma(r) \text{\, for any \,}r\in R.
\end{equation}
We also discuss the relation between many other possible forms of coassociativity. As expected, we show that the equation \eqref{coass-R} is true for any $r\in M(R)$. For our purpose we need the second assumption on $\Gamma$ which is given as
$$(\Gamma\o id)(P)=(1\o E)(P\o 1).$$
Note that the above equation is similar with the second assumption of weak multiplier Hopf algebras, i.e., $$(\Delta\o id)E=(id\o \Delta)E=(1\o E)(E\o 1)=(E\o 1)(1\o E).$$  We will give several equivalent descriptions about the second assumption, see Proposition \ref{second}. With the two assumptions, we get the notion of comodule algebra. Our definition of comodule algebras provides a unified framework that encompasses both the weak Hopf algebra and the multiplier Hopf algebra cases. 
 
In Section 4, we study the coactions of weak algebraic quantum groups, by which we mean weak multiplier Hopf algebras with {\itshape enough} integrals. The main result in this section is the duality theorem for actions and coactions. We will show that for a weak algebraic quantum group $A$ and its dual $\widehat{A}$, there is a one-to-one correspondence between the actions of $A$ and the coactions of $\widehat{A}$. 
We obtain the well known Blattner-Montgomery duality theorem for a weak algebraic quantum group pairing $\left\langle A, B \right\rangle$.  Let $R$ be a $A$-module algebra, then we can form the smash product $R\#A$ which was studied in \cite{zhou2017duality}. Let $R$ be a $B$-comodule algebra with a similar RL-condition as found in \cite{montgomery1993hopf}, then we have
$$(R\# A)\#
B\cong R\o (A\# B).$$

In Section 5, we construct a Morita context connecting the coinvariant algebra and the smash product algebras. Let $A$ be a weak algebraic quantum group with a unimodular integral and $R$ an $A$-comodule algebra. Then 
$$\{R\#\hat{A}, R^{\text{co} A}, R, R, [\cdot,\cdot],(\cdot,\cdot) \}$$
is a Morita context. This Morita context generalizes all  the Morita contexts that existed for actions of finite-dimensional Hopf algebra, for coactions of multiplier Hopf algebras and coactions of weak Hopf algebras.

In the last section, we focus on the basic examples, i.e., the infinite groupoid algebra $\mathrm{C}G$ and its dual algebra $\mathrm{C}(G)$. We will study the groupoid actions, groupoid coactions, groupoid gradings, and smash products. We will show that gradings and groupoid actions are dual concepts when $G$ is an infinite groupoid. Let $A$ be a nondegenerate algebra. We show that a true action of $G$ on $A$ is equivalent to $A$ being a $\mathrm{C}G$-module algebra, a grading of $A$ by $G$ is equivalent to $A$ being a $\mathrm{C}G$-comodule algebra or $\mathrm{C}(G)$-module algebra. When $A$ is graded, we may have a 
smash product $A\# \mathrm{C}(G)$, then we can form a Morita context. We use the groupoid algebra and its dual to illustrate the Blattner-Montgomery duality theorem obtained in Section 5.
\vspace{2 mm}
 
{\bfseries Notations.}
Let $A$ be an algebra over $\mathbb{C}$, with or without identity. The algebra $A$ is called  idempotent if $A^2=A$. And $A$ is called nondegenerate if the product is nondegenerate as a bilinear form. We denote by $A^{op}$ the algebra obtained from $A$ by reversing the product. The range and kernel of a map $f$ are denoted by $\text{ran} (f)$ and $\ker (f)$, respectively.
	
Assume that $A$ is a nondegenerate algebra. For any $a,b\in A$, a left multiplier of $A$ is a linear map $\rho_l: A\mapsto A$ such that $\rho_l(ab)=\rho_l(a)b$. A right multiplier of $A$ is a linear map $\rho_r: A\mapsto A$ such that $\rho_r(ab)=a\rho_r(b)$. A multiplier $\rho\in M(A)$ is a pair$(\rho_l,\rho_r)$ of a left and a right multiplier such that $a\rho_l(b)=\rho_r(a)b$. The algebra $A$ acts faithfully on both sides of its multiplier algebra $M(A)$. Specifically, an element $m \in M(A)$ is zero if either $am = 0$ for all $a \in A$ or $ma = 0$ for all $a \in A$. The tensor product $A\otimes A$ is again an idempotent nondegenerate algebra, so we can consider the multiplier algebra $M(A\otimes A)$. We use $1$ for the identity in any multiplier algebras.
	
The space of all linear functionals on any vector space $V$ is denoted by $V'$. Let $A$ be a nondegenerate algebra and $\omega$ a linear functional on $A$. Then $\omega$ is called {\itshape left faithful} if, given $a\in A$, we have $a=0$ if $\omega(aa')=0$ for all $a'\in A$. $\omega$ is called {\itshape right faithful} if, given $a\in A$, we have $a=0$ if $\omega(a'a)=0$ for all $a'\in A$.  A linear functional $\omega$ is called {\itshape faithful} if it is both left faithful and right faithful.
	
We will say that an algebra $A$ has local units if for any finite number of elements $a_i\in A( i=1,2,\cdots,n)$, there exists an element $e\in A$ such that $a_ie=ea_i=a_i$ for all $i$. Obviously, the property is trivial for an algebra with identity.
	
Let A be any ring. Given $a\in A$, if there exists $b\in A$ such that $aba=a, bab=b$, then we call $b$ the {\itshape generalized inverse} of $a$. The generalized inverse is not necessarily unique.
	
For a coproduct $\Delta$, we will use the Sweedler notation. In various places of \cite{van2015weak,van2020weak}, the authors explained how the use of the Sweedler notation can be made precise. In this paper, we always give an explanation when a new form of Sweedler notation is used.
	
We use $\tau$ for the flip map. The {\itshape leg numbering} notation is used. For example, let $\Delta$ be the coproduct on $A$, then $ \Delta^{13}: A \rightarrow  M(A \o A \o A)$ is defined by
$$\Delta^{13}(a)(b \o c \o d)= (id  \o \tau)(\Delta (a)(b \o d) \o c)$$
and
$$(b \o c \o d)\Delta^{13}(a)= (id \o \tau)((b \o d)\Delta (a) \o c).$$ 
	
{\bfseries Basic references.} For the theory of Hopf algebras, we refer to the books \cite{radford2011hopf} and \cite{montgomery1993hopf}. Multiplier Hopf algebras were introduced by Van Daele in \cite{van1994multiplier}. For integrals on multiplier Hopf algebras, we refer to \cite{van1998algebraic}. Weak Hopf algebras have been studied in \cite{bohm1999weak,bohm2000weak} and more results can be found in \cite{nill1998weak,nikshych2002structure,nikshych2002finite}. The theory of weak multiplier Hopf algebras was developed by Van Daele and Wang in a series of papers \cite{van2012weak,van2015weak,van2020weak,van2024duality}. Weak multiplier bialgebras were introduced by B$\ddot{\text{o}}$hm etc. in \cite{BGL}. The relation between weak multiplier Hopf algebras and multiplier Hopf algebroids as studied by T. Timmermann and A. Van Daele in papers \cite{timmermann2016integration,timmermann2017duality,timmermann2018multiplier}  was discussed in \cite{tvw2022pairing}. For the theory of actions of weak multiplier Hopf algebras, we refer to \cite{zhou2017duality,wang2018semidirect,zhou2024algebraic}. The (co)action and Galois theory of multiplier Hopf algebras, we refer to \cite{drabant1999actions,de2009galois,van1999galois,van1999corepresentation}. The explanation of the Sweedler notation for multiplier Hopf algebras can be found in \cite{van2008tools}. For the coaction theory of weak Hopf algebras, we refer to \cite{bohm2000doi,caenepeel2000modules}. We also mention that the Galois theory for weak Hopf algebras was studied by Caenepeel and De Groot in \cite{caenepeel2007galois}. And their methods are based on the Galois theory for corings.  
	
\section{Preliminaries}

\subsection{Groupoid}

One definition of a groupoid is that a groupoid is a small abstract category with inverses. The groupoid also admits an algebraic definition, which provides a natural generalization of a group. In this paper, we adopt the algebraic definition of a groupoid given in \cite{renault2006groupoid}.

\begin{definition}
A groupoid is a set $G$ endowed with a product map $(x, y)\mapsto xy: G^2\rightarrow G$ where $G^2$ is a subset of $G\times G$ called the set of composable pairs, and an inverse map $x\mapsto x^{-1}: G\rightarrow G$ such that the following relations are satisfied:
    \begin{enumerate}[(i)]
	\item $(x^{-1})^{-1}=x$,
	\item $(x, y), (y, z)\in G^2\implies (xy, z), (x, yz)\in G^2$ and $(xy)z=x(yz)$,
	\item $(x^{-1}, x)\in G^2$ and if $(x, y)\in G^2$, then $x^{-1}(xy)=y$,
	\item $(x, x^{-1})\in G^2$ and if $(z, x)\in G^2$ then $(zx)x^{-1}=z$.
    \end{enumerate}
\end{definition}

For $x\in G$, the {\itshape source} and {\itshape target} maps are defined by $\s
(x)=x^{-1}x$ and $\t(x)=xx^{-1}$, respectively. Then $(x, y)\in G^2$ iff $\s(x)=\t(y)$. We define $G^0=\s(G)=\t(G)$ to be the {\itshape unit space} of $G$. The elements of $G^0$ are called units, which satisfy the identity $x\s(x)=x=\t(x)x$. The following results are taken from \cite{ramsay1971virtual}. 

\begin{definition}\label{groupoidaction-1}
Let $G$ be a groupoid, let $S$ be a set, and let $F\subseteq G\times S$ be such that $p\in G$ implies that there is an $s\in S$ with $(x, s)\in F$. Then for $x\in G$, let $X_p=\{s\in S: (p, s)\in F\}$. A function $\gamma: F\rightarrow S, (p,s)\mapsto \gamma_p(s)$ is a {\itshape left action of $G$ in S} if the following two conditions hold:
    \begin{enumerate}[(i)]
        \item If $e$ is a unit in $G$, $s\in X_e$ implies $\gamma_e(s)=s$.
        \item  If $(p, q)\in G^2$ and $s\in X_q$, then $\gamma_q(s)\in X_p, s\in X_{pq}$ and $\gamma_p(\gamma_q(s))=\gamma_{pq}(s)$.
    \end{enumerate} 
If $S=\bigcup\{X_p: p\in G\}$, then we say that $G$ {\itshape acts on} $S$.  
\end{definition}

If $S$ is a set, let $G=\text{Inj}(S)$ denote the set of one-to-one functions from subsets of $S$ into $S$. Then $\text{Inj}(S)$ is a groupoid. And $G^2=\{(f, g)\mid D(f)=R(g)\}$, $D(f)$ is the domain of $f$ and $R(f)=f(D(f))$. If $i_A$ denotes the identity function on $A$, then $\{i_A\mid A\subseteq S\}$ is the set of units of $\text{Inj}(S)$. 

If $G_1$ and $G_2$ are two groupoids, a {\itshape homomorphism} $\gamma: G_1\rightarrow G_2$ is a function such that if $pq$ is defined then $\gamma(p)\gamma(q)$ is defined and$\gamma(pq)=\gamma(p)\gamma(q)$. By \cite[Theorem 1.1.]{ramsay1971virtual}, we know that $\gamma$ is a homomorphism of $G$ into $\text{Inj}(S)$ iff $\gamma_p$ is a left action of $G$ in $S$ for any $p\in G$. 

An action $\gamma$ is called {\itshape true} if it satisfies the Ehresmann's condition: $s\in X_p$ and $\gamma_p(s)\in X_q$ imply $pq$ is defined. We call a homomorphism $\gamma$ of groupoids {\itshape true} if $pq$ is defined whenever $\gamma(p)\gamma(q)$ is defined.

In this paper, we mainly consider an action of a groupoid on an algebra $R$. Recall that in \cite{yamanouchi1993duality}, Yamanouchi studied the action and coaction of a finite groupoid in the von Neumann algebras setting. Following \cite[Definition 2.3]{yamanouchi1993duality}, the definition of an action of a groupoid is given below.

\begin{definition}\label{groupoidaction-2}
An action of $G$ {\itshape in} an algebra $A$ consists of a family of subalgebras $\{A_e\}_{e\in G^0}$ and of isomorphisms $\{\alpha_p\}_{p\in G}$,  $\alpha_p: A_{\s(p)}\rightarrow A_{\t(p)}$ with $\alpha_{pq}=\alpha_p\alpha_q$ whenever $(p,q)\in G^2$. We write $(G,\{A_e\}_{e\in G^0},\{\alpha_p\}_{p\in G})$ for this system. If $A=\bigcup\{A_e: e\in G^0\}$, then we say that $G$  acts on $A$.
\end{definition}

Note that in \cite{paques2014duality}, Paques and Flores also give a definition of the action of a groupoid on a ring $R$ which arises from the notion of a partial groupoid action as introduced in \cite{bagio2012partial}. Their definition is essentially the same as the Definition \ref{groupoidaction-2}.
	
\subsection{ The definition of weak multiplier Hopf algebras}
	
In this subsection, we recall from \cite{van2012weak,van2015weak,van2020weak}  the definition and the basic properties of weak multiplier Hopf algebras. We also recall the definitions of the module algebra and the smash product from \cite{zhou2017duality}. 
	
\begin{lemma}\cite[Lemma 2.1]{van2015weak}\label{1}
Let $X$ be a vector space and $t$ be a linear map from $X$ to itself. Assume that $e$ and $f$ are idempotent maps on $X$ so that $e$ projects on the range of $t$ and $1-f$ projects on the kernel of $t$. Then there is a unique linear map $r$ from $X$ to itself so that 
$$tr\xi =e\xi \quad and \quad rt\xi=f\xi$$
for all $\xi\in X$. Moreover, we have $trt=t$ and $rtr=r$.	
\end{lemma}

\begin{lemma}\cite[Proposition A.3]{van2012weak}\label{extension}
Let $A$ and $B$ be non-unital algebras with non-degenerate products, and let $F$ be a homomorphism from $A$ to $M(B)$. Assume that there is an idempotent $e\in M(B)$ such that
$$F(A)B=eB \quad \text{ and} \quad BF(A)=Be.$$
Then there is a unique homomorphism $\bar{F}: M(A)\rightarrow M(B)$ such that
$\bar{F}(1)=e$ and $\bar{F}(a)=F(a)$ for all $a\in A$.  
\end{lemma}

A homomorphism $F:A\to M(B)$ with the property in Lemma \ref{extension} is called {\itshape weakly nondegenerate}.
A {\it coproduct} on an  algebra $A$ is a homomorphism $\Delta: A\rightarrow M(A\o A)$ such that
\begin{itemize}[noitemsep]
	\item $\Delta(a)(1\o b)$ and $(a\o 1)\Delta(b)$ are in $A\o A$ for all $a, b\in A$,
	\item $\Delta$ is coassociative in the sense that
	$$
	(c\o 1\o 1)(\Delta\o id)(\Delta(a)(1\o b))=(id\o\Delta)((c\o 1)\Delta(a))(1\o 1\o b)
	$$
	for all $a, b, c\in A$.
\end{itemize}
The coproduct $\Delta$ is called \textit{full} if the smallest subspaces $V$ and $W$ of $A$ satisfying
$$\Delta(A)(1\o A)\subseteq V\o A \quad \quad \text{and} \quad \quad (A\o 1)\Delta(A)\subseteq A\o W$$
are equal to $A$. The coproduct $\Delta $ is called \textit{regular} if $\Delta(a)(b\o 1)$ and $(1\o a)\Delta(b)$ are in $A\o A$ for all $a, b\in A$.

\begin{notation}
The four canonical maps $T_i(i=1,2,3,4)$ from $A\o A$ to $M(A\o A)$ are defined as
\begin{gather*}
T_1(a\o b)=\Delta(a)(1\o b),\\
T_2(a\o b)=(a\o 1)\Delta(b),\\
T_3(a\o b)=(1\o b)\Delta(a),\\
T_4(a\o b)=\Delta(b)(a\o 1).
\end{gather*}
$T_i$ is called {\itshape regular} if it has the range in $A\o A$. In this paper, we also need the following maps:
\begin{gather*}
T_1^{\tau}(a\o b)=\Delta(b)(1\o a),\\
T_2^{\tau}(a\o b)=(b\o 1)\Delta(a),\\
T_3^{\tau}(a\o b)=(1\o a)\Delta(b),\\
T_4^{\tau}(a\o b)=\Delta(a)(b\o 1).
\end{gather*}
Clearly, $T_i^{\tau}=T_i\circ \tau$  for $i=1,2,3,4$. If $T_i$ is regular, so is $T_i^{\tau}$.
\end{notation}

Let $A$ be an algebra with a coproduct $\Delta$, define
$\Delta^{\text{cop}}: A\rightarrow M(A\o A)$
via
$$
\Delta ^{\text{cop}}(a)(b \o c):= \tau(\Delta (a)(c \o b))\, \mbox { and }\,  (b \o c)\Delta ^{\text{cop}}(a):= \tau((c \o b)\Delta (a)).
$$
$A^{\text{cop}}$ is the algebra $A$ with the coproduct $\Delta^{\text{cop}}$.  

A linear map $\varepsilon: A\rightarrow \mathbb{C}$ is called a {\it counit}
on an algebra $A$ with a coproduct $\Delta$ if
$$
(\varepsilon\o id)(\Delta(a)(1\o b))=ab \quad \text{and} \quad (id \o \varepsilon)((a\o 1)\Delta(b))=ab
$$
for all $a, b\in A$. 
	
Now we recall the definition of weak multiplier Hopf algebras from \cite{van2015weak}.

\begin{definition}\label{wmha}
	A weak multiplier Hopf algebra is a pair $(A,\Delta)$ of a nondegenerate idempotent algebra with a full coproduct and a counit satisfying the following conditions.
	
	(i) There exists an idempotent $E\in M(A\otimes A)$ that gives the ranges of the canonical maps $T_1$ and $T_2$:
	$$
	T_1(A\otimes A)=E(A\otimes A)\quad {and} \quad T_2(A\otimes A)=(A\otimes A)E.
	$$
	
	(ii) The element $E$ satisfies
	$$
	(id \otimes \Delta)E=(E\otimes 1)(1\otimes E)=(1\otimes E)(E\otimes 1),
	$$
	
	(iii) The kernels of $T_1$ and $T_2$ are of the form
	$$
	\text{ker}(T_1)=(1-G_1)(A\otimes A),$$
	$$
	\text{ker}(T_2)=(1-G_2)(A\otimes A),
	$$
	where $G_1$ and $G_2$ are linear maps from $A\otimes A$ to itself, given as
	$$
	(G_1\otimes id)(\Delta^{13}(a)(1\otimes b\otimes c))=\Delta^{13}(a)(1\otimes E)(1\otimes b\otimes c),
	$$
	$$
	(id\otimes G_2)((a\otimes b\otimes 1)\Delta^{13}(c))=(a\otimes b\otimes 1)(E\otimes 1)\Delta^{13}(c),
	$$
	for all $a, b, c\in A$.
\end{definition}

The idempotent $E$ in the above definition is called the {\itshape canonical idempotent}. It is uniquely determined and satisfies $\Delta(a)E=\Delta(a)=E\Delta(a)$ for all $a\in A$. By Lemma \ref{extension}, the coproduct can be extended in a unique way to a homomorphism from $M(A)$ to $M(A\o A)$ such that $\Delta(1)=E$. Similarly, we have an extension for $\Delta \o id$ and $id\o \Delta$. Then the coassociativity of the coproduct can be expressed as $(\Delta\o id)\Delta=(id\o \Delta)\Delta$.

Since the coproduct is full, any element in $A$ is a linear span of elements of the form $(id\o \omega)(\Delta(a)(1\o b))$ or $(\omega\o id)((c\o 1)\Delta(a))$ where $a, b, c\in A, \omega\in A'$. If the coproduct is regular, we have a similar result.  The fullness of the coproduct implies that the counit is unique. For the proof, we refer to \cite[Lemma 1.11]{van2012weak}. The notation {\itshape full} was first introduced by Van Daele and Wang, see \cite[Proposition 1.6]{van2006larson}.

We will use WMHA as a shorthand notation for weak multiplier Hopf algebra. 

Given a WMHA $(A, \Delta, E, \varepsilon)$, there is a unique \textit{antipode} $S$. It is a linear map from $A$ to $M(A)$. It is an anti-algebra map in the sense that $S(ab)=S(b)S(a)$ for any $a, b\in A$. We also know that $S$ is nondegenerate (meaning that $AS(A)=A$ and $S(A)A=A$). By Lemma \ref{extension}, the antipode has a unique extension.  The antipode satisfies the formulas
$$\sum a_{(1)}S(a_{(2)})a_{(3)}=a \quad \text{and} \quad \sum S(a_{(1)})a_{(2)}S(a_{(3)})=S(a)$$
for all $a\in A$. For an explanation of the Sweedler notation in these equations, we refer to \cite[Remark 2.5]{van2015weak}.
	
A WMHA is called \textit{regular} if the coproduct is regular and if also $(A, \Delta^{cop})$ is a weak multiplier Hopf algebra. A WMHA is regular if and only if the antipode $S$ maps $A$ to $A$ and is bijective, see \cite[Theorem 4.10]{van2015weak}. In the regular case, the antipode is an anti-coalgebra map in the sense that $\Delta(S(a))=\tau(S\o S)\Delta(a)$ for all $a\in A$. 

Let $(A, \Delta, E, S, \varepsilon)$ be a WMHA. The \textit{source} and \textit{target} maps $\varepsilon_s: A\rightarrow M(A)$ and $\varepsilon_t: A\rightarrow M(A)$ are defined by 
$$\varepsilon_s(a)=\sum S(a_{(1)})a_{(2)} \quad \text{and} \quad \varepsilon_t(a)=\sum a_{(1)}S(a_{(2)}).$$
These maps are well-defined in $M(A)$, see \cite[Remark 2.5]{van2015weak}. The images $\varepsilon_s(A)$ and $\varepsilon_t(A)$ are commuting nondegenerate subalgebras of $M(A)$. They are called the \textit{source algebra} and the \textit{target algebra}, respectively. Van Daele and Wang also show that any WMHA has local units, see \cite[Proposition 2.21]{van2020weak}. Denote by $A_s$ the set of elements $y\in M(A)$ satisfying $\Delta(y)=E(1\o y)$ and by $A_t$ the set of elements $x\in M(A)$ satisfying $\Delta(y)=E(x\o 1)$. In the regular case, we have $M(\varepsilon_s(A))=A_s$ and $M(\varepsilon_t(A))=A_t$. For any $x\in A_t,y\in A_s$, we have
\begin{equation}\label{xE}
(1\o x)E=(S(x)\o 1)E,
\end{equation}
\begin{equation}\label{Ey}
E(y\o 1)=E(1\o S(y)).
\end{equation}

Given a WMHA $(A, \Delta, E, S, \varepsilon)$, the generalized inverses of $T_1$ and $T_2$ are given by:
\begin{gather*}
R_1(a\o b)=\sum a_{(1)}\o S(a_{(2)})b,\\
R_2(a\o b)=\sum aS(b_{(1)})\o b_{(2)}.
\end{gather*}
Note that the antipode $S$ is determined by $R_1$ and $R_2$, i.e., the existence of $R_1$ and $R_2$ induce the existence of $S$. If $S$ is bijective, then we can define 
\begin{gather*}
R_3(a\o b)=\sum a_{(1)}\o bS^{-1}(a_{(2)}),\\
R_4(a\o b)=\sum S^{-1}(b_{(1)})a\o b_{(2)}.
\end{gather*}
In the regular case, we  have four idempotent maps  
$$F_1=(id\o S)E, \quad \quad F_2=(S\o id)E, \quad\quad F_3=(id\o S^{-1})E, \quad \quad F_4=(S^{-1}\o id)E.$$
The four elements are multipliers of $A\o A^{\text{op}}$(or $A^{\text{op}}\o A$). And recall that 
\begin{equation}\label{G_i}
G_i(a\o b)=(a\o 1)F_i(1\o b),\quad i=1,2,3,4.
\end{equation}
The integrals of WMHA will be recalled in Section 4. 	
	
\subsection{Module algebra and smash product}

Let $A$ be a  WMHA. By a left $A$-module, we mean a vector space $R$ equipped with a bilinear map $\rhd: A\otimes R\rightarrow R, a\otimes r\mapsto ar$ satisfying $(ab)\rhd r=a(b\rhd r)$ for all $r\in R$ and $ a, b\in A$. The module is called {\itshape unital } if $AR=R$. The module is called nondegenerate if $r\in R$ and $ar=0$ for all $a\in A$, imply $r=0$. Since $A$ has local units, a unital module is nondegenerate. Given a unital $A$-module $R$, there is a unique extension to a left $M(A)$-module with $1\rhd r=r$.

\begin{definition}\label{modulealgebra}
Let $A$ be a  WMHA and $R$ an algebra. Then $R$ is called a left $A$-\textit{module algebra} if $R$ is a left unital $A$-module and satisfies
\begin{equation}\label{modalg}
a\rhd (ru)=\sum (a_{(1)}\rhd r)(a_{(2)}\rhd u)
\end{equation}
for any $r,u\in R, a\in A$.
\end{definition}
The definition is well-defined, see \cite{zhou2017duality}. We give some useful formulas. For any $x\in A_t, y\in A_s, u\in R$,
\begin{equation}\label{xru}
x\rhd (ru)=(x\rhd r)u,
\end{equation}
\begin{equation}\label{yru}
(y\rhd r)u=r(S(y)\rhd u). 
\end{equation}
If $R$ has a unit $1_R$, then $a\rhd 1_R=\varepsilon_t(a)\rhd 1_R$. When $A$ is a regular WMHA and $R$ is a left $A$-module algebra, by \cite[Proposition 3.4]{zhou2017duality}, we can extend the action of $A$ from $R$ to $M(R)$ by 
\begin{equation}\label{right-r}
  (a\rhd m)r=\sum a_{(1)}\rhd (m(S(a_{(2)})\rhd r))  
\end{equation}
\begin{equation}\label{left-r}
 r(a\rhd m)=\sum a_{(2)}\rhd ((S^{-1}(a_{(1)})\rhd r)m)   
\end{equation}
for any $a\in A, m\in M(R), r\in R$. Now $a\rhd 1_{M(R)}$ is meaningful and we have $a\rhd 1_{M(R)}=\varepsilon_t(a)\rhd 1_{M(R)}$. 

Let $R$ be a left $A$-module algebra, then we have a smash product algebra $R\# A=E\rhd (R\o A)$ with the product
$$(r\# a)(r'\# a')= \sum r(a_{(1)}\rhd r')\# r'_{(2)}a'.$$
The product is associative and nondegenerate. The main result in \cite{zhou2017duality} is the duality theorem for actions and their dual actions on the smash product of weak algebraic quantum groups. Specifically, given a weak algebraic quantum group pairing $\left\langle A, \widehat{A}\right\rangle$ and a left $A$-module algebra $R$. We have the following algebra isomorphism: 
$$(R\# A)\# \widehat{A}\cong R\o(A\bar{\o} \widehat{A}).$$
We will illustrate this result in Section 6 where $A$ is the groupoid algebra.   
	
\section{Coaction and comodule algebra}	
\subsection{The definition of coaction}	
Let $(A, \Delta, E, S, \varepsilon)$ be a WMHA and $R$  an algebra with a nondegenerate product. In the following, elements in $R$ and $A$ will be denoted by $r, r'$ and $a, a'$, respectively. 
	
\begin{notation}
Given a linear map $\Gamma: R \rightarrow M(R\o A)$, we have the following four maps 
$$ W_1: R\o A\rightarrow M(R\o A), \quad r\o a\mapsto \Gamma(r)(1\o a),$$
$$ W_2: R\o R\rightarrow M(R\o A), \quad r\o r'\mapsto (r\o 1)\Gamma(r'),$$
$$ W_3: R\o A\rightarrow M(R\o A), \quad r\o a\mapsto (1\o a)\Gamma(r),$$
$$ W_4: R\o R\rightarrow M(R\o A), \quad r\o r'\mapsto \Gamma(r')(r\o 1).$$		
\end{notation}

$W_1$ and $W_3$ are called the \textit{canonical maps} associated with $\Gamma$. $W_2$ and $W_4$ are called the \textit{Galois maps}. The map $W_i$ ($i = 1, 2, 3, 4$) is called \textit{regular} if its range is contained in $R \otimes A$. 
		 	 
\begin{proposition}\label{id-delta}
The map $id\o \Delta$ has a unique extension to a homomorphism $(id\o \Delta)_1$ from $M(R\o A)$ to $M(R\o A\o A)$, satisfying $(id\o \Delta)_1(1)=1\o E$.
\end{proposition}
\begin{proof}
For any $m\in M(R\o A)$. Denote $E(a\o a')= \sum\Delta(p_i)(1\o q_i)$ and $(a\o a')E=\sum (u_i\o 1)\Delta(v_i)$. Then we can define $(id\o \Delta)_1$ by
$$((id\o \Delta)_1(m))(r\o a\o a')=(id\o \Delta)(m(r\o p_i))(1\o 1\o v_i),$$
$$(r\o a\o a')((id\o \Delta)_1(m))=\sum (1\o u_i\o 1)(id\o \Delta)((r\o v_i)m).$$
This is well-defined and it is standard to show that $(id\o \Delta)_1(m)$ is a multiplier. We also refer to the proof of Proposition 1.18 in \cite{van2024reflections1}. 
\end{proof}

As usual, we will denote the extension $(id\o \Delta)_1$ by $id\o \Delta$. Remark that in the above proof we do not assume that $R$ is idempotent. If $R$ is idempotent, then $id\o \Delta$ is weakly nondegenerate and the above result can be deduced from the Lemma \ref{extension}.
	
\begin{definition}
Let $A$ be a WMHA and $R$ a nondegenerate algebra. If $\Gamma$ is a linear map from $R$ to $M(R\o A)$, then $\Gamma$ is called coassociative if 
\begin{enumerate}[(i)]
    \item $W_1$ is regular, i.e. $\Gamma(r)(1\o a)\in R\o A$ for all $r\in R$ and $a\in A$.
    \item For any $r\in R, a\in A$, we have
        \begin{equation}\label{coass1}
        (\Gamma\o id)(\Gamma(r)(1\o a))=((id\o \Delta)\Gamma(r))(1\o 1\o a).
	\end{equation}
\end{enumerate}
\end{definition}
By Proposition \ref{id-delta}, the formula $(id\o \Delta)\Gamma(r)$ is meaningful. But it is too early to expect that $(\Gamma\o \text{id})\Gamma(r)$ is meaningful since we cannot extend the map $\Gamma$. Right now, the regularity of $W_1$ is needed for the left side of equation (\ref{coass1}) to be meaningful.

\begin{remark}
\begin{enumerate}[(1)]
\item The equation \eqref{coass1} is also used in the definition of  coaction of multiplier Hopf algebras, see \cite[Definition 2.1]{van1999galois}. The so-called \textit {counit property} will be discussed later. Since $\varepsilon$ is not a nondegenerate homomorphism, we cannot define $(id\o \varepsilon)\Gamma(r)$.  We also cannot regarded it as a multiplier, since we do not require $W_2$ or $W_4$ to be regular. 
    
\item There are several possible ways to define the coassociativity of a coaction. Some of them are of a very different nature, within different situations. 
		  \begin{enumerate}[(a)]			
		\item  If $W_3$ is regular, we call that $\Gamma$ is coassociative if 
		      \begin{equation}\label{coass2}
		      (\Gamma\o id)((1\o a)\Gamma(r))=(1\o 1\o a)((id\o \Delta)\Gamma(r)).
		      \end{equation}
		      Here we use the extension of the map $id\o \Delta$ to $M(R\o A)$.
              
		\item  If $W_1$ and $W_2$ are regular,  then equation (\ref{coass1}) can be reformulated as 
		      \begin{equation}\label{coass3}
		      (W_2\o id)(id\o W_1)=(id\o T_1)(W_2\o id).
		      \end{equation}
		
		\item  If $W_1$ and $W_4$ are regular, then equation (\ref{coass1}) can be reformulated as  
		\begin{equation}\label{coass1-2}
		(W_4\o id)(id\o W_1)=(id\o T_1)(W_4\o id).
		\end{equation}
				
		\item If $W_2, W_3$ and $T_3$ are regular, then equation (\ref{coass2}) can be reformulated as 
		\begin{equation}\label{coass1-4}
		(W_2\o id)(id\o W_3)=(id\o T_3)(W_2\o id).
		\end{equation}
				
		\item If $W_3, W_4$ and $T_3$ are regular, then equation (\ref{coass2}) can be reformulated as 
		\begin{equation}\label{coass1-5}
		(W_4\o id)(id\o W_3)=(id\o T_3)(W_4\o id).
		\end{equation}	
		The expressions $\eqref{coass3}-\eqref{coass1-5}$ are both defined on $R\o R\o A$, and their ranges are in $R\o A\o A$. 	
		\end{enumerate}
		For more information about the coassociativity of the coaction map on a non-unital algebra, we refer to \cite{van2024reflections1}.
\end{enumerate}
\end{remark}
	
\begin{definition}
Let $A$ be a WMHA and $R$ a nondegenerate algebra. Assume that $\Gamma$ is a linear map from $R$ to $M(R\o A)$. We call $\Gamma$ {\itshape semi-regular} if $W_1$ and $W_3$ are regular.
If all four maps $W_i (i=1,2,3,4)$ are regular, then we call $\Gamma$ {\itshape regular}.
\end{definition}
	
By the nondegeneracy of the product, we obtain the following results.

\begin{proposition}\label{coass-eqv1}
\begin{enumerate}[(i)]
    \item  If $W_1$ and $W_3$ are regular, then equation \eqref{coass1} $\Leftrightarrow$ equation \eqref{coass2}.
    \item  If $\Delta$ and $\Gamma$ are regular, then the coassociativity conditions \eqref{coass1}--\eqref{coass1-5} are all equivalent.
\end{enumerate}
\end{proposition}

Since we are  interested in the notion of comodule algebra, we need $\Gamma$ to be a homomorphism. When $\Gamma$ is a homomorphism, we have another notion of coassociativity.

\begin{definition}
Let $A$ be a WMHA and $R$ a nondegenerate algebra. If $\Gamma$ is a homomorphism from $R$ to $M(R\o A)$, then we call $\Gamma$ coassociative if
\begin{equation}\label{coass4}
(\Gamma\o id)(\Gamma(r)(r'\o a))=((id\o \Delta)\Gamma(r))(\Gamma(r')\o a)
\end{equation}
and
\begin{equation}\label{coass4'}
(\Gamma\o id)((r'\o a)\Gamma(r))=(\Gamma(r')\o a)((id\o \Delta)\Gamma(r)).
\end{equation}
\end{definition}

\begin{proposition}\label{coass-eqv2}
\begin{enumerate}[(1)]
    \item If $\Gamma$ is  homomorphism and $W_1$ is regular, then equation \eqref{coass1} $\Rightarrow$ equation \eqref{coass4}. Moreover, if $W_1$ is surjective, then the converse is true.
    \item If $\Gamma$ is  homomorphism and $W_3$ is regular, then equation $\eqref{coass2}\Rightarrow \text{equation } \eqref{coass4'}$. If $W_3$ is surjective, then the converse is true.
\end{enumerate}
\end{proposition}
\begin{proof}
(1) Since $A$ has local units, it follows that for any given $a$ there exists an element $e$ in $A$ such that $ae=a$. Then
\begin{align*}
    (\Gamma\o id)(\Gamma(r)(r'\o a))
    &=(\Gamma\o id)(\Gamma(r)(1\o a))(\Gamma(r')\o e)\\
    &=((id\o \Delta)\Gamma(r))(1\o 1\o a)(\Gamma(r')\o e)\\
    &=((id\o \Delta)\Gamma(r))(\Gamma(r')\o a).
\end{align*}
Conversely. Take any $b\in A$, we have
\begin{align*}
    (\Gamma\o id)(\Gamma(r)(r'\o a))(1\o b\o 1)=
    (\Gamma \o id)(\Gamma(r)(1\o a))(\Gamma(r')(1\o b)\o 1)
\end{align*}
and
$$((id\o \Delta)\Gamma(r))(\Gamma(r')\o a)(1\o b\o 1)=(id\o \Delta)\Gamma(r)(1\o 1\o a)(\Gamma(r')(1\o b)\o 1).$$
Since $W_1$ is surjective, we can cancel $\Gamma(r')(1\o b)$ and get the equation $\eqref{coass1}$.
	 
(2) Similarly.
\end{proof}

In the case of (regular multiplier) Hopf algebras, $W_1$ and $W_3$ are bijective, see \cite[Proposition 2.5]{van1999galois}. However, in the theory of weak multiplier Hopf algebras, these maps are generally not surjective. Requiring both maps to be surjective would be too restrictive. We will show that the converse result holds under a weaker hypothesis.

\begin{definition}
Let $\Gamma$ be a homomorphism from $R$ to $M(R\o A)$. Assume that $W_1$ is regular. $\Gamma$ is called a \textit{right coaction} of $A$ on $R$ if $\Gamma$ is coassociative in the sense of equation \eqref{coass1}. If $\Gamma$ is (semi-)regular, then the coaction is called (semi-)regular.
\end{definition}

\begin{remark}
\begin{enumerate}[(1)]
\item In the theory of multiplier Hopf algebras \cite{van1999galois}, the authors used the term {\itshape reduced coaction} for the regular coaction defined above. Later, we will see that the assumptions on the coaction map are similar to those on the coproduct of a WMHA. To emphasize this connection, we adopt the term regular rather than reduced.
\item In general, we should first consider the comodule of a WMHA $A$ over a vector space $V$. There are two ways to consider this notion. With the notion of completed module introduced by Van Daele in \cite{van2008tools}, the coaction of $A$ on $V$ is a linear map $\Gamma: A\rightarrow M_0(V\o A)$ satisfying $(\Gamma\o id)\Gamma=(id\o \Delta)\Gamma$ where $M_0(V\o A)$ is the completion of the $A$-bimodule $V\o A$. The equation is meaningful without extra assumptions. In another way, the comodule could be defined via a compatible pair of maps $\lambda,\rho: V\o A\rightarrow V\o A$ (instead of a map $V\rightarrow V\o A$), see \cite{van1999corepresentation,kurose2000corepresentation,bohm2014comodules}. We will also arrive at the compatible conditions of the pair $(\lambda,\rho)$, see Proposition \ref{compatible-pair}.
\end{enumerate}
\end{remark}

If $W_2$ is regular and
\begin{equation}\label{counit1}
(id\o \varepsilon)((r'\o 1)\Gamma(r))=r'r,
\end{equation}
then $\Gamma$ is called satisfying the counit property. If $W_4$ is regular and 
\begin{equation}\label{counit2}
(id\o \varepsilon)(\Gamma(r)(r'\o 1))=rr',
\end{equation}
then $\Gamma$ is called satisfying the counit property. If $W_2$ and $W_4$ are regular, by the nondegeneracy of $R$, the equations \eqref{counit1} and \eqref{counit2} are equivalent. If $W_2$ and $W_4$ are regular, we may denote the counit property as $(id\o \varepsilon)\Gamma(r)=r$, since $(id\o \varepsilon)\Gamma(r)$ can be viewed as a multiplier in the natural way.
 
\begin{definition}
\begin{enumerate}[(i)]
\item Assume that $W_1$ is regular. Denote by $V_1$ the smallest subspace of $R$ satisfying $\Gamma(r)(1\o a)\in V_1\o A$. We will call $V_1$ the left leg of $\Gamma$.
\item  Assume that $W_2$ is regular. Denote by $V_2$ the smallest subspace of $A$ satisfying $(r'\o 1)\Gamma(r)\in R\o V_2$. We will call $V_2$ the right leg of $\Gamma$.
\item Assume that $W_3$ is regular. Denote by $V_3$ the smallest subspace of $R$ satisfying $(1\o a)\Gamma(r)\in V_3\o A$. We will call $V_3$ the left leg of $\Gamma$.
\item  Assume that $W_4$ is regular. Denote by $V_4$ the smallest subspace of $A$ satisfying $\Gamma(r)(r'\o 1)\in R\o V_4$. We will call $V_4$ the right leg of $\Gamma$.
\end{enumerate}
$\Gamma$ is called \textit{left full} if the left leg is $R$. $\Gamma$ is called \textit{right full} if the right leg is $A$. The coaction $\Gamma$ is called \textit{full} if it is both left full and right full.
\end{definition}
	
If $W_1$ and $W_3$ are regular, then $V_1=V_3$. If $W_2$ and $W_4$ are regular, then $V_2=V_4$. 

\begin{proposition}
\begin{enumerate}[(1)]
\item If $W_1$ is regular then $V_1=\{(id\o \omega)(\Gamma(r)(1\o a))\mid \omega\in A', r\in R, a\in A\}$,
\item If $W_2$ is regular then $V_2=\{(\eta\o id)(r'\o 1)\Gamma(r)\mid \eta\in R', r, r'\in R\}$,
\item If $W_3$ is regular then $V_3=\{(id\o \omega)((1\o a)\Gamma(r))\mid \omega\in A', r\in R, a\in A\}$,
\item If $W_4$ is regular then $V_4=\{(\eta\o id)(\Gamma(r)(r'\o 1))\mid \eta\in A', r,r'\in R\}$.
\end{enumerate}
\end{proposition}
\begin{proof}
The proof is simple and we refer to a similar proof of Proposition 1.6 in \cite{van2024reflections3}.
\end{proof}

\subsection{First properties about $W_1$ and $W_3$}

Now we assume that $\Gamma$ is a right coaction of $A$ on $R$. Then $W_1$ behaves with respect to multiplication as
$$W_1(id\o \mu)=(id\o \mu)(W_1\o id)$$
where $\mu$ denotes multiplication in $A$. The equation above means that
\begin{equation*}
W_1(r\o aa')=W_1(r\o a)(1\o a'),
\end{equation*}
for all $r\in R$ and $ a,a'\in A$.
We may have the following equation
\begin{equation}\label{zz}
(\Gamma\o id)W_1=(id\o T_1)(\Gamma\o id).
\end{equation}
Note that it is not safe to use equation $\eqref{zz}$ since the right side is meaningless. Equation $\eqref{zz}$ is meaningful only in the case where $W_2$ or $W_4$ is regular. For example, if $W_2$ is regular, then the expression
$$(t\o 1\o 1)(id\o T_1)(\Gamma\o id)(r\o a)$$
is meaningful. Thus, the equation $\eqref{zz}$ is nothing else but the coassociativity equation $\eqref{coass3}$. If $W_3$ is regular, we have $W_3(r\o aa')=(1\o a)W_3(r\o a)$. 

\begin{proposition}\label{compatible-pair}
Let $\Gamma$ be a right coaction of $A$ on $R$. We have
    \begin{equation}\label{RR-corep}
    (T_1)^{23}(W_1)^{12}=(W_1)^{12}(W_1)^{13}(T_1)^{23}.
    \end{equation}
If $W_3$ is regular, then we have
    \begin{equation}\label{RL-corep}
    (T_2)^{23}(W_3)^{13}=(W_3)^{12}(W_3)^{13}(T_2)^{23}.
    \end{equation}
If $\Gamma$ and $\Delta$ are regular, then we have
    \begin{equation}\label{12-23}
    (W_j)^{12}(W_j)^{23}=(T_j)^{23}(W_j)^{13}(W_j)^{12}, \quad j=2,4.
    \end{equation}
\end{proposition}
\begin{proof}
Note that the two sides of equations (\ref{RR-corep}) and (\ref{RL-corep}) are linear maps from $R\o A\o A$ to $R\o A\o A$, and the two sides of equation (\ref{12-23}) are linear maps from $R\o R\o R$ to $R\o A\o A$. The proof is easy.
\end{proof}

Remark that equations \eqref{RR-corep} and \eqref{RL-corep} were used in \cite{van1999corepresentation} to define the right-right corepresentation and the right-left corepresentation of multiplier Hopf algebras. B\"ohm also uses these equations to define the comodule over a regular multiplier bialgebra, see the coassociativity conditions in \cite[Definition 2.1]{bohm2014comodules}. Obviously, we have
$$(1\o a)W_1(r\o b)=(W_3(r\o a))(1\o b)$$
with any $r\in R, a,b\in A$. This is the compatible condition in \cite[Definition 2.7]{van1999corepresentation} or \cite[Definition 2.1]{bohm2014comodules}. 

More formulas will be obtained and we summarize them as follows. Considering the two equations
\begin{equation}\label{eq12}
x^{23}y^{12}=y^{12}y^{13}x^{23},
\end{equation}
\begin{equation}\label{eq13}
x^{23}y^{13}=y^{12}y^{13}x^{23},
\end{equation}
$x$ is a linear map from $A\o A$ to $A\o A$, $y$ is a linear map from $R\o A$ to $R\o A$.  If $x=a$ and $ y=b$ satisfy the two equations, then we denote the solution by a pair $(a, b)$. 
 
\begin{proposition}
Assume that $\Gamma$ is a right regular coaction of $A$ on $R$. If $A$ is regular, then 
\begin{enumerate}[(1)]
    \item $(T_1, W_1),(T_4^{\tau}, W_1) (T_3, W_3), (T_2^{\tau}, W_3)$ are solutions of equation \eqref{eq12}.
    \item $(T_1^{\tau}, W_1), (T_4, W_1), (T_2, W_3), (T_3^{\tau}, W_3)$ are solutions of equation \eqref{eq13}.
\end{enumerate}
\end{proposition}

\begin{remark}\label{cover-1}		
Assume that $\Gamma$ is a right coaction of $A$ on $R$.	
			 
(1) Since  $\Gamma(r)(1\o a)\in R\o A$, we may use Sweedler notation  
$$\Gamma(r)(1\o a)=\sum r_{(0)}\o r_{(1)}a.$$
And we will say that $r_{(1)}$ is {\itshape covered} by $a$. Moreover, we will use the following expression
\begin{equation*}\label{coass}
(id\o \Delta)\Gamma(r)=\sum r_{(0)}\o r_{(1)}\o r_{(2)}.
\end{equation*}
The right side is meaningful if we multiply with $1\o a\o b$ from the right or left. Specifically, if we denote $E(a\o b)=\sum_i \Delta(a_i)(1\o b_i)$, then we have
$$\sum r_{(0)}\o r_{(1)}a\o r_{(2)}b=\sum r_{(0)}\o (r_{(1)}a_i)_{(1)}\o (r_{(1)}a_i)_{(2)}b_i.$$
Now $r_{(1)}$ is covered by $a_i$, then $(r_{(1)}a_{i})_{(2)}$ is covered by $b_i$. 
	
(2) We consider the expression $(id\o S)\Gamma(r)$. Recall that $AS(A)=A=S(A)A$, we can denote $a=pS(q)$, then
$$\sum r_{(0)}\o aS(r_{(1)})=\sum r_{(0)}\o pS(r_{(1)}q)\in R\o A.$$
So $(id\o S)\Gamma(r)$ is a right multiplier of $R\o A$. 		
\begin{itemize}
	 \item If $W_2$ is regular, then we have $\sum r'r_{(0)}\o S(r_{(1)})\in R\o A$. And $(id\o S)\Gamma(r)$ is still a right multiplier.
	 \item If $W_3$ is regular, denote $a=S(p)q$, then 
	 	$$\sum r_{(0)}\o S(r_{(1)})a=\sum r_{(0)}\o S(pr_{(1)})q\in R\o A.$$
	 	So $(id\o S)\Gamma(r)$ is a left multiplier. Hence $(id\o S)\Gamma(r)\in M(R\o A)$. And we can consider the following important map $$\overline{W_1}: R\o A\rightarrow R\o A, r\o a\mapsto \sum r_{(0)}\o S(r_{(1)})a.$$
	 \item If $W_4$ is regular, then $\sum r_{(0)}r'\o S(r_{(1)})\in R\o A$. So $(id\o S)\Gamma(r)\in M(R\o A)$. 
\end{itemize}
To define the multiplier $(id\o S)\Gamma(r)$, we require the regularity of $W_3$ or $W_4$. We will choose $W_3$ due to Proposition \ref{13}. If $A$ is regular and now $S$ is bijective, then we can consider the multiplier $(id\o S^{-1})\Gamma(r)$ in a similar way. If $W_3$ and $A$ are regular, we can define the map $\overline{W_3}$ as follows,
$$\overline{W_3}: R\o A\rightarrow R\o A, r\o a \mapsto \sum r_{(0)}\o aS^{-1}(r_{(1)}).$$ 
By the Sweedler notation, we denote the multiplier $(id\o S)\Gamma(r)$ as $\sum r_{(0)}\o S(r_{(1)})$ and $(id\o S^{-1})\Gamma(r)$ as $\sum r_{(0)}\o S^{-1}(r_{(1)})$.
	
(3) Next, we focus on the expression $ (id\o \varepsilon_t)\Gamma(r)$. 
\begin{itemize}
    \item If $W_2$ and $W_4$ are regular, we have
            $$(id\o \varepsilon_t)((r'\o 1)\Gamma(r))\in R\o \varepsilon_t(A) \quad \text{and} \quad (id\o \varepsilon_t)(\Gamma(r)(r'\o 1))\in R\o \varepsilon_t(A).$$
            Then we can view $ (id\o \varepsilon_t)\Gamma(r)$ as a multiplier of $R\o A$.
    \item In the following we will assume that $W_3$ is regular. Let us consider the range of $W_1\overline{W_1}$, we have 
            $$W_1\overline{W_1}(r\o a)=\sum r_{(0)}\o \varepsilon_t(r_{(1)})a \in R\o A.$$
            Then we can say that $r_{(1)}$ is covered by $a$. Note that $(id\o \varepsilon_t)\Gamma(r)$ is only a left multiplier because we can not obtain that $\sum r_{(0)}\o a\varepsilon_t(r_{(1)}) \in R\o A$ or $\sum r'r_{(0)}\o \varepsilon_t(r_{(1)}) \in R\o A$. Denote the element $\sum r_{(0)}\o \varepsilon_t(r_{(1)})a$ as $\sum u_i\o v_i$ with $u_i\in R, v_i\in A$, then we find
            $$\sum r_{(0)}\o S(r_{(1)})r_{(2)}S(r_{(3)})a=\sum u_{i(0)}\o S(u_{i(1)})v_i\in R\o A.$$
            By the property of $S$, the left side is equal to $\sum r_{(0)}\o S(r_{(1)})a$. Similarly, we could consider the  expression $(id\o \varepsilon_s)\Gamma(r)$ and we have
            $$ \overline{W_1}W_1(r\o a)=\sum r_{(0)}\o \varepsilon_s(r_{(1)})a \in R\o A.$$
            When $A$ is regular, we can define the right multipliers $(id\o \varepsilon'_t)\Gamma(r)$ and $(id\o \varepsilon'_s)\Gamma(r)$ as
            \begin{align*}
                W_3\overline{W_3}(r\o a)=\sum r_{(0)}\o a\varepsilon'_t(r_{(1)})\in R\o A,\\
                \overline{W_3}W_3(r\o a)=\sum r_{(0)}\o a\varepsilon'_s(r_{(1)})\in R\o A.
            \end{align*}
\end{itemize} 
\end{remark}

\begin{notation}
Denote $Q=\overline{W_1}W_1$ and $ U=\overline{W_3}W_3$. 
\end{notation}

\begin{proposition}\label{13}
Let $\Gamma$ be a right semi-regular coaction of $A$ on $R$. Then the map $\overline{W_1}$ satisfies
$$\overline{W_1}W_1\overline{W_1}=\overline{W_1} \quad \text{and} \quad W_1\overline{W_1}W_1=W_1.$$
If $A$ is regular, then the map $\overline{W_3}$ satisfies 
$$\overline{W_3}W_3\overline{W_3}=\overline{W_3} \quad \text{and} \quad W_3\overline{W_3}W_3=W_3.$$
In this sense, $\overline{W_1}$ and $\overline{W_3}$ are called the generalized inverses of  $W_1$ and $W_3$, respectively.
\end{proposition}
\begin{proof}
For any $r\in R, a\in A$, we have
\begin{align*}
\overline{W_1}W_1\overline{W_1}(r\o a)&=\sum r_{(0)}\o S(r_{(1)})r_{(2)}S(r_{(3)})a\\
&=\sum r_{(0)}\o S(r_{(1)})a\\
&= \overline{W_1}(r\o a).
\end{align*}
The proofs for the other three equations are similar.
\end{proof}

\begin{proposition}\label{behave}
Let $\Gamma$ be a right semi-regular coaction of $A$ on $R$. Assume that $W_i(i=2 \text{\, or \,}4)$ is regular. The following hold:
\begin{enumerate}[(1)]
	\item $\overline{W_1}(r\o aa')=\overline{W_1}(r\o a)(1\o a')$,
	\item $(W_i\o id)(id\o \overline{W_1})=(id\o R_1)(W_i\o id)$ on $R\o R\o A$.
\end{enumerate}
If $A$ is regular, then we have
\begin{enumerate}[(a)]
	\item $\overline{W_3}(r\o aa')=(1\o a)\overline{W_3}(r\o a')$,
	\item $(W_i\o id)(id\o \overline{W_3})=(id\o R_3)(W_i\o id)$ on $R\o R\o A$.
\end{enumerate}
\end{proposition}
\begin{proof}
The proof is easy.
\end{proof}

For the maps $Q$ and $U$, we have the following results.

\begin{proposition}\label{c3}
Let $\Gamma$ be a right semi-regular coaction of $A$ on $R$. Assume that $W_i(i=2 \text{\, or \,}4)$ is regular. The following equations hold.
\begin{enumerate}[(1)]
\item $Q(r\o aa')=Q(r\o a)(1\o a')$,
\item $(id\o \Delta)Q=Q^{13}(id\o \Delta)$,
\item $(W_i \o id)(id\o Q)=(id\o G_1)(W_i \o id)$.
\end{enumerate}
If $A$ is regular, then we have
\begin{enumerate}[(a)]
\item $U(r\o a'a)=(1\o a')U(r\o a),$
\item $(id\o \Delta)U=U^{12}(id\o \Delta),$
\item $(W_i\o id)(id\o U)=(id\o G_3)(W_i \o id)$.
\end{enumerate}
\end{proposition}
\begin{proof}
The proof is easy.
\end{proof}

\begin{proposition}\label{kernel}
Let $\Gamma$ be a right semi-regular coaction of $A$ on $R$. We have
$$\ker(W_1)=(1-Q)(R\o A).$$
If $A$ is regular, then
$$\ker(W_3)=(1-U)(R\o A).$$
\end{proposition}
\begin{proof}
Let us consider the first equation. Since $W_1\overline{W_1}W_1=W_1$, we have $\ran(1-\overline{W_1}W_1)\subseteq \ker(W_1)$. Conversely, assume that there exists $x\in \ker(W_1)$, but $x\notin \ran(1-\overline{W_1}W_1)$. Indeed, we have $x=(1-\overline{W_1}W_1)x\in \ran(1-\overline{W_1}W_1)$. So we get $\ker(W_1)\subseteq \ran(1-\overline{W_1}W_1)$. Then we prove the first equation. The second one is similar.
\end{proof}

Since we are working with (weak) multiplier Hopf algebras, it is natural and necessary to consider the ranges of $W_1$ and $W_3$. This is what we are doing in the next subsection.
	
\subsection{The extension of coaction map $\Gamma$}
	
Let $\Gamma$ be a right semi-regular coaction of $A$ on $R$. Note that we do not assume that $A$ is regular and $R$ is idempotent. We will start with the following \textit{familiar} assumption.
	
{\bfseries Assumption I.}
Assume that there exists an idempotent element $P\in M(R\o A)$ such that
\begin{equation}\label{P}
P(R\o A)=\Gamma(R)(1\o A) \quad \text{and} \quad (R\o A)P=(1\o A)\Gamma(R).
\end{equation}
	
If $R$ is idempotent, then $\Gamma$ is weakly non-degenerate in the sense of that
$$P(R\o A)=\Gamma(R)(R\o A) \quad \text{and} \quad (R\o A)P=(R\o A)\Gamma(R).$$

\begin{remark}
In fact, we have another choice about the first assumption. If we consider the axioms of WMHA, then we should assume that
\begin{equation}\label{P'}
P(R\o A)=\Gamma(R)(1\o A) \quad \text{and} \quad (R\o A)P=(R\o 1)\Gamma(R).
\end{equation} 
Note that in the coaction theory of multiplier Hopf algebras \cite{van1999galois}, a non-regular multiplier Hopf algebra $(A, \Delta)$ is not an $A$-comodule algebra with the coaction $\Delta$. If we use Assumption I, we still cannot include the examples of non-regular WMHAs. However, if we replace equation \eqref{P} with equation \eqref{P'}, we find that the non-regular (weak) multiplier Hopf algebras do fit into the theory. As we introduced before, $W_2$ is a Galois map and we will study the map in the context of Galois theory for WMHAs.  
\end{remark}

In this subsection, we will always assume that Assumption I is satisfied. 
	
\begin{definition}
Let $D$ be an algebra with a nondegenerate product. For any idempotents $m,n\in M(D)$, we call $m$ is {\itshape smaller} than $n$ if $m=mn=nm.$
\end{definition}
	
\begin{proposition}
The element $P$ is the smallest idempotent in $M(R\o A)$ such that
\begin{equation}\label{3}
P\Gamma(r)=\Gamma(r) \quad \text{and} \quad \Gamma(r)P=\Gamma(r).
\end{equation}
\end{proposition}
\begin{proof}
By assumption, for any $a\in A, r, r'\in A$, we have 
$$P(\Gamma(r)(1\o a))=\Gamma(r)(1\o a).$$
By the non-degeneracy of the product, we can cancel $a$ and obtain that $P\Gamma(r)=P$. Similarly for the second equation. Assume that $P'$ is another idempotent that satisfies equation \eqref{3}. Denote $P(r'\o a')=\sum \Gamma(r_i)(1\o a_i)$, then we have
$$P'P(r'\o a')=\sum P'\Gamma(r_i)(1\o a_i)=\sum \Gamma(r_i)(1\o a_i)=P(r'\o a'),$$
it implies that $P'P=P$. Similarly we can get $PP'=P$. This means that $P\leq P'$.
\end{proof}

\begin{corollary}\label{Pisunique}
If an idempotent satisfies Assumption I, then it is unique.
\end{corollary}

Now we will call that $P$ is the \textit{canonical idempotent} for the map $\Gamma$.
	
\begin{proposition}
Assume that $\Gamma$ satisfies the counit property. If $r\in R$ and if $(r\o 1)P=0$ or $P(r\o 1)=0$, then $r=0$.
\end{proposition}
\begin{proof}
Take any $r\in R$. If $P(r\o 1)=0$ then $\Gamma(r')(r\o 1)=0$ for all $r'\in R$. Now we apply  counit on the second leg then we find $rr'=0$. By the non-degeneracy of $R$, we get $r=0$.
\end{proof}
	
\begin{proposition}\label{extension-gamma}
$\Gamma$ can be uniquely extended to the map from $M(R)$ to $M(R\o A)$ such that $\Gamma(1)=P$. Here we denote the extension map still by $\Gamma$.
\end{proposition}
\begin{proof}
For any $m\in M(R), r\o a\in R\o A$,  we have $P(r\o a)=\sum \Gamma(r_i)(1\o v_i)$ with $r_i, v_i\in A$. We can define
$$\Gamma(m)(r\o a)=\Gamma(m)P(r\o a)=\sum  \Gamma(mr_i)(1\o v_i).$$
Similarly, we can define $(r\o a)\Gamma(m)$. The rest of the proof is standard.
\end{proof}

In a completely similar way, we can extend $\Gamma\o id$ to the map from $M(R\o A)$ to $M(R\o A\o A)$ such that 
$$(\Gamma\o id)(1)=P\o 1.$$
The $1$ in the left side is the identity in $M(R\o A)$ and the right side 1 is the identity in $M(A)$. We also use the following natural embeddings:
$$M(R\o A)\o M(A)\subseteq M(R\o A\o A)$$
and
$$M(R)\o M(A\o A)\subseteq M(R\o A\o A).$$
So $P\o 1$ and $1\o E$ can be regarded as elements in $M(R\o A\o A)$. Now the $1$ in $1\o E$ is the identity in $M(R\o A)$.
	
Since $(\Gamma \o id)\Gamma(r)$ is meaningful, we get the following result.
	
\begin{proposition}\label{2.8}
$(\Gamma \o id)\Gamma(r)=(id\o \Delta)\Gamma(r)$ for all $r\in R$.
\end{proposition}
\begin{proof}
We use $\tilde{\Gamma}\o id$ for the extension of $\Gamma \o id$. Firstly, we show that
$$((\tilde{\Gamma}\o id)\Gamma(r))(1\o 1\o a)=(\Gamma\o id)(\Gamma(r)(1\o a))$$
for any $r\in R, a\in A$. Let $b\in A$ and $x\in R\o A\o A$, we can denote 
$$x(P\o 1)=\sum u_i(\Gamma \o id)(v_i)$$
with $u_i\in R\o A\o A$ and $v_i\in R\o A$. Then
		\begin{align*}
			x((\tilde{\Gamma}\o id)\Gamma(r))(1\o 1\o b)&=\sum u_i(\Gamma \o id)(v_i)((\tilde{\Gamma}\o id)\Gamma(r))(1\o 1\o b)\\
			&= \sum u_i(\Gamma \o id)(v_i\Gamma(r))(1\o 1\o b)\\
			&= \sum u_i(\Gamma\o id)(v_i\Gamma(r)(1\o b))\\
			&= \sum u_i(\Gamma \o id)(v_i)(\Gamma \o id)(\Gamma(r)(1\o b))\\
			&= x(\Gamma \o id)(\Gamma(r)(1\o b)).
		\end{align*}
Next,by the coassociativity equation \eqref{coass1}, we get
$$((\tilde{\Gamma}\o id)\Gamma(r))(1\o 1\o a)=(id\o \Delta)(\Gamma(r))(1\o 1\o a)$$
and this proves the result.
\end{proof}

\begin{corollary}
For any $i\in \{1,3\}$, $W_i^{12}W_i^{13}=W_i^{12}W_i^{13}E^{23}=P^{12}W_i^{12}W_i^{13}$.
\end{corollary}

Let $i=1$, we get $W_1^{12}W_1^{13}=W_1^{12}W_1^{13}E^{23}$. Together with equation \eqref{RR-corep}, we obtain the normalization conditions in \cite[Definition 2.1]{bohm2014comodules}. 

\begin{corollary}
Let $\Gamma$ be a regular coaction with the Assumption I satisfied. If $\Gamma$ is injective, then $\Gamma$ satisfies the counit property.
\end{corollary}
\begin{proof}
For any $r,r', t\in R, a\in A$, we have 
	\begin{align*}
	\Gamma((id\o \varepsilon)((r'\o 1)\Gamma(r)))(t\o a)
	&=(id\o id\o \varepsilon)(\Gamma\o id)((r'\o 1)\Gamma(r)(t\o a\o 1))\\
	&=(\Gamma(r')\o 1)(id\o id\o \varepsilon)((\Gamma\o id)\Gamma(r)(t\o a\o 1))\\
	&=(\Gamma(r')\o 1)(id\o id\o \varepsilon)((id\o \Delta)\Gamma(r)(t\o a\o 1))\\
	&=\Gamma(r'r)(t\o a).
	\end{align*}
By injectivity of $\Gamma$, we obtain $(id\o \varepsilon)((r'\o 1)\Gamma(r))=r'r$. Similarly, we can show that $(id\o \varepsilon)(\Gamma(r)(r'\o 1))=rr'$.
\end{proof}

\begin{corollary}\label{coass-eqv3}
The equation \eqref{coass4} implies the equation \eqref{coass1}.
\end{corollary}
\begin{proof}
From the proof of Proposition \ref{coass-eqv2}, we have 
$$(\Gamma \o id)(\Gamma(r)(1\o a))(P\o 1)(u\o v\o 1)=(id\o \Delta)\Gamma(r)(1\o 1\o a)(P\o 1)(u\o v\o 1)$$
for any $u\in R, v\in A$. Since $P\o 1=(\Gamma\o id)(1\o 1)$, the left-hand side equals $(\Gamma \o id)(\Gamma(r)(1\o a))(u\o v\o 1)$. 
By Proposition \ref{2.8}, we have $(id\o \Delta)\Gamma(r)(P\o 1)=(id\o \Delta)\Gamma(r)$. So
$$(\Gamma \o id)(\Gamma(r)(1\o a))(u\o v\o 1)=(id\o \Delta)\Gamma(r)(1\o 1\o a)(u\o v\o 1).$$
Now we can cancel $u, v$ and get equation $\eqref{coass1}$.
\end{proof}

If $\Gamma$ is a regular coaction with the Assumption I satisfied, then the different coassociativity equations(i.e. equations \eqref{coass1} to \eqref{coass1-5}, equations \eqref{coass4} and \eqref{coass4'}) are equivalent. 

\begin{proposition}\label{gammaPdelta}
We have $(\Gamma \o id)(P)=(id\o \Delta)(P)$.
\end{proposition}
\begin{proof}
For any $r, a,b$, we have
	\begin{align*}
	(\Gamma \o id)(P)(r\o a\o b)
	&=(\Gamma \o id)(P)(P\o 1)(r\o a\o b)\\
	&= (\Gamma \o id)(P)(\Gamma \o id)(\eta\o b)(1\o \delta\o 1) \tag*{denote $P(r\o a)=W_1(\eta\o \delta)$}\\
	&= (\Gamma \o id)(P((\eta\o b)))(1\o \delta \o 1)\\
	&=( (\Gamma \o id)\Gamma(e))(P\o 1)(1\o \delta \o f) \tag*{denote $P(\eta \o b)=W_1(e\o f)$}\\
	&= (\Gamma \o id)\Gamma(e)(1\o \delta \o f)\\
	&= (id\o \Delta)\Gamma(e)(1\o \delta \o f).
	\end{align*}
This means that $(id\o \Delta)\Gamma(R)(1\o A\o A)=(\Gamma \o id)(P)(R\o A\o A)$. Similarly, we can show that 
	$$(1\o A\o A)(id\o \Delta)\Gamma(R)=(R\o A\o A)(\Gamma \o id)(P).$$
So $(\Gamma \o id)(P)$ is the canonical idempotent for the map $(id\o \Delta)\Gamma$.  On the other hand, we have 
	\begin{align*}
	(id\o \Delta)(P)(r\o a\o b)
	&=(id\o \Delta)(P)(1\o E)(r\o a\o b)\\
	&=(id\o \Delta)(P)(r\o \Delta(u)(1\o v)) \tag*{denote $E(a\o b)=T_1(u\o v)$}\\
	&=(id\o \Delta)(P(r\o u))(1\o 1\o v)\\
	&=(id\o \Delta)(\Gamma(m)(1\o n))(1\o 1\o v) \tag*{denote $P(r\o u)=W_1(m\o n)$}\\
	&=(id\o \Delta)\Gamma(m)(1\o \Delta(n)(1\o v))\\
	&=(id\o \Delta)\Gamma(m)(1\o \xi\o \kappa) \tag*{denote $E(\xi\o \kappa)=T_1(n\o v)$}\\
	&=(\Gamma\o id)\Gamma(m)(1\o \xi\o \kappa).
	\end{align*}
Then we obtain $(id\o \Delta)\Gamma(R)(1\o A\o A)=(id\o \Delta)(P)(R\o A\o A)$. Similarly, we can get
	$$(1\o A\o A)(id\o \Delta)\Gamma(R)=(R\o A\o A)(id \o \Delta)(P).$$
So $(id\o \Delta)(P)$ is another canonical idempotent for $(id\o \Delta)\Gamma$. Then by the uniqueness of canonical idempotent, we obtain $(\Gamma\o id)(P)=(id\o \Delta)(P)$.
\end{proof}

\begin{corollary}\label{general-coass}
We have
\begin{equation*}
(\Gamma \o id)\Gamma(m)=(id\o \Delta)\Gamma(m) 
\end{equation*}
for all $m\in M(R)$.
\end{corollary}
\begin{proof}
The result follows from the equation in Proposition \ref{gammaPdelta} and the uniqueness in Proposition \ref{extension-gamma}.
\end{proof}

\begin{notation}
We denote $(\Gamma \o id)(P)$ or $(id\o \Delta)(P)$ by $\mathcal{H}$.
\end{notation}

\begin{remark}
By the definition we have $(\Gamma\o id)(1)=P\o 1$, then
$$(P\o 1)\mathcal{H}=\mathcal{H}(P\o 1)=\mathcal{H}.$$
This means that $\mathcal{H}\leq P\o 1$. Similarly by $(id\o \Delta)(1)=1\o E$, we have
$$(1\o E)\mathcal{H}=\mathcal{H}(1\o E)=\mathcal{H}.$$
Then $\mathcal{H}\leq 1\o E$. So we get an idempotent that is smaller than both $P\o 1$ and $1\o E$. Moreover,  we have
$$\mathcal{H}\leq (P\o 1)(1\o E)$$
and
$$\mathcal{H}\leq (1\o E)(P\o 1).$$
Inspired by the theory of WMHA in \cite{van2012weak, van2015weak}, we may require that $P\o 1$ and $1\o E$ commute and that 
$$\mathcal{H}=(P\o 1)(1\o E).$$
Next we will give a detailed discussion about the equation. 
\end{remark}

We continue the study of the maps $W_1$ and $W_3$. We assume that $\Gamma$ is a right semi-regular coaction of $A$ on $R$ with the idempotent $P$ satisfying Assumption I.  $P_l$ is left multiplication with $P$ and $P_r$ is right multiplication with $P$. By Proposition \ref{kernel} and Assumption I,  the range and kernel of $W_1$ can be described by
$$\ran(W_1)=P_l(R\o A),\quad  \ker(W_1)=(1-Q)(R\o A).$$
So $P_l$ and $Q$ are idempotent maps on $R\o A$ satisfying that $P$ projects on the range of $W_1$ and $1-Q$ projects on the kernel of $W_1$. Then by Lemma \ref{1}, we get the following result.

\begin{proposition}\label{5}
The pair of equations 
$$W_1x=P_l, \quad xW_1=Q$$ 
has a unique solution, denoted by $\widetilde{W_1}$. If $A$ is regular, then the pair of equations
$$W_3y=P_r,\quad  yW_3=U$$
has a unique solution, denoted by $\widetilde{W_3}$. Moreover, $\widetilde{W_1}$ and $\widetilde{W_3}$ are generalized inverses of $W_1$ and $W_3$, respectively. And the two generalized inverses satisfy that 
$$\widetilde{W_1}(r\o aa')=(\widetilde{W_1}(r\o a))(1\o a') \quad\text{and} \quad \widetilde{W_3}(r\o aa')=(1\o a)(\widetilde{W_3}(r\o a')).$$
\end{proposition}
\begin{proof}
We only need to check the last pair of equations.
	\begin{align*}
	\widetilde{W_1}(r\o aa') &=(\widetilde{W_1}W_1\widetilde{W_1})(r\o aa')\\
				  &= \widetilde{W_1}((W_1\widetilde{W_1})(r\o a)(1\o a'))\\
				  &= \widetilde{W_1}W_1(\widetilde{W_1}(r\o a)(1\o a'))\\
				  &=(\widetilde{W_1}W_1\widetilde{W_1}(r\o a))(1\o a')\\
				  &=\widetilde{W_1}(r\o a)(1\o a').
	\end{align*}
We use the property for $W_1\widetilde{W_1}=P_l$ in the second equality, then for $W_1$ and finally for $\widetilde{W_1}W_1=Q$. The rest is similar.
\end{proof}

Note that the generalized inverse is not unique, so we can not conclude that $\widetilde{W_1}$ is equal to $\overline{W_1}$ and $\widetilde{W_3}$ is equal to $\overline{W_3}$. We need more conditions and it is time to introduce the second assumption on $P$.

\subsection{The second assumption and main definition}

Let $A$ be a WMHA and $R$ a non-degenerate algebra. In this subsection, we always assume that $(\Gamma,P)$ is a  semi-regular coaction with the Assumption I satisfied. 

\begin{lemma}\label{P_lP_r}
Let $(\Gamma,P)$ be a semi-regular coaction of $A$ on $R$.  
\begin{enumerate}[(1)]
	\item If $P(r\o 1)=(id\o \varepsilon_t)\Gamma(r)$, then $\widetilde{W_1}=\overline{W_1}$. 
	\item  If $A$ is regular and $(r\o 1)P=(id\o \varepsilon'_t)\Gamma(r)$, then $\widetilde{W_3}=\overline{W_3}.$
\end{enumerate}
\end{lemma}
\begin{proof}
(1) By assumption, we have 
\begin{equation}\label{6}
P(r\o a)=W_1\overline{W_1}(r\o a)=\sum r_{(0)}\o \varepsilon_t(r_{(1)})a.
\end{equation}
By the definition of $Q$, we have $\overline{W_1}W_1=Q$. By the uniqueness in Proposition \ref{5}, we must have $\widetilde{W_1}=\overline{W_1}$.
	
(2) Similarly.
\end{proof}

\begin{lemma}\label{commute1}
Let $(\Gamma,P)$ be a semi-regular coaction of $A$ on $R$. Assume that $W_2$ is regular. If $P(r\o 1)=(id\o \varepsilon_t)\Gamma(r)$, then we have 
\begin{equation}\label{commute2}
(1\o E)(P\o 1)=(P\o 1)(1\o E).
\end{equation}
If $A$ is regular and $(r\o 1)P=(id\o \varepsilon'_t)\Gamma(r)$, we also get the equation \eqref{commute2}.
\end{lemma}
\begin{proof}
Note that $(id\o \varepsilon_t)\Gamma(r)$ is an element in $M(R\o A)$. For any $t', t\in R$,  since $W_2$ is regular, we have 
\begin{equation}\label{t'Pt}
(t'\o 1)P(t\o 1)=\sum t't_{(0)}\o \varepsilon_t(t_{(1)})\in R\o \varepsilon_t(A).
\end{equation}
We use notation $P=P_{(0)}\o P_{(1)}$ and $E=E_{(1)}\o E_{(2)}$. Then
\begin{align*}
&(t\o 1\o 1)(1\o E)(P\o 1)(r\o a\o b)\\
&=tP_{(0)}r\o E_{(1)}P_{(1)}a\o E_{(2)}b\\
&= tP_{(0)}r\o P_{(1)}E_{(1)}a\o E_{(2)}b\\
&=(t\o 1\o 1)(P\o 1)(1\o E)(r\o a\o b)
\end{align*}
By formula \eqref{t'Pt}, the second and third equalities are meaningful. In the third equality, we use the fact that the source algebra commutes with the target algebra. Then we prove the formula.

Since $\varepsilon_t$ and $\varepsilon'_t$ have the same ranges, the proof of the rest is similar.
\end{proof}

\begin{proposition}\label{1EP1}
Let $(\Gamma,P)$ be a semi-regular coaction of $A$ on $R$. Assume that $W_2$ is regular. If $P(r\o 1)=(id\o \varepsilon_t)\Gamma(r)$, then we have
\begin{equation}\label{commute3}
\mathcal{H}=(1\o E)(P\o 1)
\end{equation}
and
\begin{equation}
\mathcal{H}=(P\o 1)(1\o E).
\end{equation}
\end{proposition}
\begin{proof}
For any $t,r\in R, x, y\in A$, we have
\begin{align*}
&(t\o 1\o 1)(id\o \Delta)(P)(r\o x\o y)\\
&=(t\o 1\o 1)(id\o \Delta)(P)(1\o E)(r\o x\o y)\\
&=(t\o 1\o 1)(id\o \Delta)(P(r\o a))(1\o 1\o b) \tag*{denote $E(x\o y)=\Delta(a)(1\o b)$}\\
&=\sum tr_{(0)}\o \Delta (\varepsilon_t(r_{(1)})a)(1\o b) \tag*{by equation \eqref{6}}\\
&= \sum tr_{(0)}\o (\varepsilon_t(r_{(1)})\o 1)E\Delta(a)(1\o b)  \tag*{by second equality of (\ref{xE})}\\
&=(t\o 1\o 1)(P\o 1)(1\o E)(r\o x\o y)
\end{align*}
Note that the regularity of $W_2$ is necessary, otherwise the fourth equality is meaningless. When $W_2$ is regular, we have equation \eqref{t'Pt}. Then we can use $\Delta(\varepsilon_t(r_{(1)}))=(\varepsilon_t(r_{(1)})\o 1)E$. So we have proved that $(id\o \Delta)(P)=(P\o 1)(1\o E)$. By Lemma \ref{commute1}, the second equation is obtained. 
\end{proof}

\begin{proposition}\label{second}
Let $(\Gamma,P)$ be a semi-regular coaction of $A$ on $R$.  Assume that $W_2$ is regular. Consider the following equations:
\begin{enumerate}[(1)]
\item \begin{enumerate}
			 \item $\mathcal{H}=(1\o E)(P\o 1),$
			 \item  $\mathcal{H}=(P\o 1)(1\o E),$
	          \end{enumerate}
\item   \begin{enumerate}
			\item $(\Gamma\o id)P_l=(id\o E_l)(\Gamma\o id)$,
			\item  $(\Gamma\o id)P_r=(id\o E_r)(\Gamma\o id)$,
			\end{enumerate}
\item    
\begin{enumerate}
			\item $(W_2\o id)(id\o \widetilde{W_1})=(id\o R_1)(W_2\o id),$
			\item  $(W_2\o id)(id\o \widetilde{W_3})=(id\o R_3)(W_2\o id),$
			\end{enumerate}
\item    \begin{enumerate}
			\item  $P(r\o 1)=(id\o \varepsilon_t)\Gamma(r)\in M(R\o A),$
			\item  $(r\o 1)P=(id\o \varepsilon'_t)\Gamma(r)\in M(R\o A).$
			\end{enumerate}
\end{enumerate}
Then we have 
\begin{enumerate}[(i)]
    \item The equations $(1a), (2a), (3a), (4a)$ are equivalent.
    \item If $A$ is regular, then all the equations are equivalent. 
\end{enumerate}
\end{proposition}
\begin{proof}
Note that $(1a)$ and $(1b)$ are multipliers in $M(R\o A\o A)$, $(2a)$ and $(2b)$ are linear maps from $R\o A$ to $M(R\o A\o A)$, $(3a)$ and $(3b)$ are linear maps from $R\o R\o A$ to $R\o A\o A$, formulas $(3b)$ and $(4b)$ are only legal when $A$ is regular. The multipliers $(id\o \varepsilon_t)\Gamma(t)$ and $(id\o \varepsilon'_t)\Gamma(t)$ have been explained in Remark \ref{cover-1}.
	
 	(i): $(1a)\implies (2a)$. For any $r\in R, a\in A$, we have
	\begin{align*}
	(\Gamma\o id)(P_l(r\o a)) &=(\Gamma\o id)(P(r\o a))\\
										  &= (\Gamma \o id)(P)(\Gamma(r)\o a)\\
										  &= (1\o E)(P\o 1)(\Gamma(r)\o a)\\
										  &=(1\o E)(\Gamma(r)\o a)\\
										  &=(id\o E_l)((\Gamma\o id)(r\o a)).
	\end{align*}
	
	$(2a)\implies (1a)$. We can apply this equation to $r\o a$ and use that $\Gamma\o id$ is a homomorphism, then we find
	$$(\Gamma\o id)(P)(\Gamma(r)\o a)=(1\o E)(\Gamma(r)\o a).$$
	This implies that 
        $$(\Gamma\o id)(P)=(\Gamma\o id)(P)(P\o 1)=(1\o E)(P\o1).$$
	
	$(2a)\implies (3a)$. Since $W_2$ is regular,  the equation $(2a)$ can be denoted as  
	\begin{equation}\label{2a'}
	(W_2\o id)(id\o P_l)=(id\o E_l)(W_2\o id).
	\end{equation}
By (ii) of Proposition \ref{c3} we have
$$(W_2\o id)(id\o Q)
=(id\o G_1)(W_2\o id).$$
Apply $id\o \widetilde{W_1}$ and use $Q\widetilde{W_1}=\widetilde{W_1}W_1\widetilde{W_1}=\widetilde{W_1}$, then
\begin{align*}
(W_2\o id)(id\o \widetilde{W_1}) 
&=(id\o G_1)(W_2\o id)(id\o \widetilde{W_1}) 
\\
&=(id\o R_1)(id\o T_1)(W_2\o id)(id\o \widetilde{W_1})\\
&=(id\o R_1)(W_2\o id)(id\o W_1)(id\o \widetilde{W_1}) \tag*{by equation \eqref{coass3}}\\
&= (id\o R_1)(W_2\o id)(id\o P_l)\\
&= (id\o R_1)(id\o E_l)(W_2\o id) \tag*{by equation \eqref{2a'}}\\
&= (id\o R_1)(W_2\o id). \tag*{by $R_1E_l=R_1T_1R_1=R_1$}
\end{align*}

$(3a)\implies (2a)$. Since $W_2$ is regular, it is equivalent to show equation \eqref{2a'}. We have
\begin{align*}
(W_2\o id)(id\o P_l)
&=(W_2\o id)(id\o W_1)(id\o \widetilde{W_1})\\
&=(id\o T_1)(W_2\o id)(id\o \widetilde{W_1}) \tag*{by equation \eqref{coass3}}\\
&=(id\o T_1)(id\o R_1)(W_2\o id) \tag*{by assumption $(3a)$}\\
&=(id\o E_l)(W_2\o id).
\end{align*}

$(3a)\implies (4a)$.  Denote $\widetilde{W_1}(t\o a)=t_{\widetilde{W_1}}\o a_{\widetilde{W_1}}$. Then the formula $3(a)$ can be expressed as 
$$\sum rt_{\widetilde{W_1}(0)}\o t_{\widetilde{W_1}(1)}\o a_{\widetilde{W_1}}=\sum rt_{(0)}\o t_{(1)}\o S(t_{(2)})a.$$
Apply the multiplication of $A$, then we get
$$\sum rt_{\widetilde{W_1}(0)}\o t_{\widetilde{W_1}(1)} a_{\widetilde{W_1}}=\sum rt_{(0)}\o \varepsilon_t(t_{(1)})a.$$
It means that $(r\o 1)W_1\widetilde{W_1}(t\o a)=\sum rt_{(0)}\o \varepsilon_t(t_{(1)})a$. Cancel $r$ and $a$, then we get 
$$P_l(t\o 1)=\sum t_{(0)}\o \varepsilon_t(t_{(1)}).$$

$(4a)\implies (2a)$. It is equivalent to show equation \eqref{2a'} and the proof is similar to the one of $(3a)\implies (2a)$. We have
\begin{align*}
(W_2\o id)(id\o P_l) &=(W_2\o id)(id\o W_1)(id\o \overline{W_1})\tag*{by Lemma \ref{P_lP_r}}\\
&=(id\o T_1)(W_2\o id)(id\o \overline{W_1}) \tag*{by equation \eqref{coass3}}\\
&=(id\o T_1)(id\o R_1)(W_2\o id) \tag*{by (ii) of Proposition \ref{behave}}\\
&=(id\o E_l)(W_2\o id).
\end{align*}
 
(ii): Firstly, we will show that $(1b)\Leftrightarrow (2b)\Leftrightarrow (3b)\implies (4b)\implies (2b)$ and the proof is completely similar with part (i). Next by Lemma \ref{commute1}, we get $(1a)\Leftrightarrow (1b)$, then we finish the proof.
\end{proof}

\begin{remark}\label{important}
In this remark, we give more information about the items (3) and (4) of Proposition \ref{second}. 
\begin{enumerate}[(i)]
        \item If we are working with (weak) Hopf algebras, then the formulas $(3a)$ and $(3b)$ are equivalent with 
                    \begin{equation}\label{3ab}
				        (\Gamma\o id)\widetilde{W_1}=(id\o R_1)(\Gamma\o id) \text{\, and\, } (\Gamma\o id)\widetilde{W_3}=(id\o R_3)(\Gamma\o id).
                    \end{equation}
		 But we want to emphasize that it is not safe to use the formulas \eqref{3ab} without the regularity of $W_2$ or $W_4$. Later we will show that $W_2$ is regular is equivalent with $W_4$ is regular. If the coaction $\Gamma$ is regular, the item (3) can be summarized as
		 \begin{equation}
		 (W_i\o id)(id\o \widetilde{W_j})=(id\o R_j)(W_i\o id), \quad i\in \{2,4\}, j\in \{1,3\}.
		 \end{equation}

	\item Note that from formula $4(a)$ we can not conclude that $P(R\o 1)\in R\o \varepsilon_t(A)$. If $R$ is idempotent and $W_4$ is regular, then we have 
		 $$P(rr'\o 1)=\sum r_{(0)}r'\o \varepsilon_t(r_{(1)}).$$
		 we find that $P(R\o 1)\in R\o \varepsilon_t(A)$. Similarly, in that case, formula $(4b)$ implies $(R\o 1)P\in R\o \varepsilon_t(A)$. Moreover, if $R$ has an identity, then we can get $P\in R\o \varepsilon_t(A)$. If $A$ is weak Hopf algebra, the condition $P\in R\o \varepsilon_t(A)$ is equivalent with the formulas in Proposition \ref{second}.
		 
	\item The equations $4(a)$ and $4(b)$ means that the right leg of $P$  belongs to the target algebra of $A$. To be precise, by equation \eqref{t'Pt} we have that       $(\omega(r'\cdot r)\o id)(P)\in \varepsilon_t (A)$ for all linear functionals $\omega$ on $R$ and all $r', r''\in R$. Moreover, if $\Gamma$ is right full, then any       element in $\varepsilon_t(A)$ is a linear combination of such elements. 
	\end{enumerate}
\end{remark}

Now we come to the main definition.

\begin{definition}
Let $\Gamma$ be a semi-regular coaction of $A$ on $R$.  $R$ is called a right $A$-\textit{comodule algebra} if there exists an idempotent $P\in M(R\o A)$ such that 
	\begin{enumerate}[(1)]
		\item $P(R\o A)=\Gamma(R)(1\o A) \quad \text{and} \quad (R\o A)P=(1\o A)\Gamma(R)$,
		\item $(\Gamma\o id)(P)=(1\o E)(P\o 1)$.
	\end{enumerate}
We denote it by $(R, \Gamma, P)$. The comodule algebra is called \textit{regular} if $\Gamma$ is regular. 
\end{definition}

\begin{proposition}
Let $R$ be an $A$-comodule algebra. If $W_2$ is regular, then the coaction $\Gamma$ is regular. Similarly, if $A$ is regular and $W_4$ is regular, then $\Gamma$ is regular.
\end{proposition}
\begin{proof}
Assume that $W_2$ is regular. We need to show that $W_4$ is regular. For any $x,y\in R$, we have
	\begin{align*}
	\Gamma(x)(y\o 1) 
	&=\Gamma(x)P(y\o 1)\\
	&=\sum \Gamma(x)(y_{(0)}\o \varepsilon_t(y_1))\\
	&=\sum \Gamma(xy_{(0)})(1\o S(y_{(1)}))\in R\o A.
	\end{align*}
When $A$ is regular, the proof is similar.
\end{proof}

Recall that $\Gamma$ satisfies the counit property iff $\Gamma$ is injective. And we will have a more direct proof of $(1a)\implies (4a)$ when $\Gamma$ is injective.

\begin{proposition}
Let $R$ be a regular $A$-comodule algebra. If $\Gamma$ is injective, then we have $P(r\o 1)=(id\o \varepsilon_t)\Gamma(r)$.
\end{proposition}
\begin{proof}
	For any $r\in R, a,b \in A$, denote $P(r\o a)=\sum \Gamma(u_i)(1\o v_i)$. Then we have
	\begin{align*}
	(t\o 1\o 1)(\Gamma\o id)(P)(r\o a\o b)
	&=\sum (t\o 1\o 1)(\Gamma\o id)(P)(\Gamma(u_i)(1\o v_i)\o b)\\
	&=\sum (t\o 1\o 1)(\Gamma\o id)(P(u_i\o 1))(1\o v_i\o b).
	\end{align*}
	On the other hand, we have
	\begin{align*}
	(t\o 1\o 1)(1\o E)(P\o 1)(r\o a\o b)
	&=\sum (t\o 1\o 1)(1\o E)(u_{i(0)}\o u_{i(1)}v_i\o b)\\
	&=\sum tu_{i(0)}\o u_{i(1)}v\o \varepsilon_t(u_{(2)})b\\
	&=\sum (t\o 1\o 1)(\Gamma\o id)((id\o \varepsilon_t)\Gamma(u_i))(1\o v_i\o b).
	\end{align*}
In the second equality, $u_{i(1)}$ is covered since $W_2$ is regular. Using the injective of $\Gamma$, we obtain $P(r\o 1)=(id\o \varepsilon_t)\Gamma(r)$.
\end{proof}

Similarly, we can introduce the left $A$-comodule algebra.

\begin{definition}\label{leftcomodulealgebra}
Let $\Gamma$ be a left regular coaction of $A$ on $R$. $R$ is called a left $A$-\textit{comodule algebra} if there exists an idempotent element $P\in M(A\o R)$ such that 
	\begin{itemize}
	\item $P(A\o R)=\Gamma(R)(A\o 1) \quad \text{and} \quad (A\o R)P=(A\o 1)\Gamma(R)$,
	\item $(id\o \Gamma)P=(E\o 1)(1\o P)$.
	\end{itemize}
\end{definition}
	
When the comodule algebra is regular and $A$ is regular, we find more results. Recall that the maps $G_i(i=1,2,3,4)$ in regular WMHA are given by idempotents in multiplier algebras, see equation \eqref{G_i} or \cite[Proposition 4.5]{van2015weak}. Now for the maps $Q$ and $U$, we have similar results. Define two idempotents $V_1$ and $V_3$:
$$V_1=(id\o S)P\quad \text{and} \quad V_3=(id\o S^{-1})P.$$
They are elements in $M(R\o A^{\text{op}})$. 
\begin{proposition}
Assume that $A$ is a regular WMHA. Let $(R, \Gamma, P)$ be a regular $A$-comodule algebra. Then we have
$$Q(r\o a)=(r\o 1)V_1(1\o a)$$
and
$$U(r\o a)=(1\o a)V_3(r\o 1).$$
\end{proposition}
\begin{proof}
For any $r\in R, S(a)\in A$, we have
        \begin{align*}
	(r\o 1)((id\o S)P)(1\o S(a))&= (id\o S)((r\o a)P)\\
	&=\sum (id\o S)(r_{(0)}\o a\varepsilon'_t(r_{(1)}))\\
	&=\sum r_{(0)}\o \varepsilon_s(r_{(1)})S(a)\\
	&=Q(r\o S(a)).
	\end{align*}
In fact, the first equality gives a correct interpretation of the formula $V_1=(id\o S)P$. Similarly for $V_3$.
\end{proof}

We also have 
$$\ker (W_1)=(R\o 1)(1- V_1)(1\o A) \text{\quad and \quad} \ker(W_3)=(R\o 1)(1- V_3)(1\o A)$$
and
$$P^{13}(V_1\o 1)=P^{13}(1\o E)\text{\quad and \quad}
(V_3\o 1)P^{13}=(1\o E)P^{13}.$$

\begin{proposition}
Assume that $A$ is a regular WMHA. Let $(R, \Gamma, P)$ be a regular $A$-comodule algebra. If $R$ is idempotent, then we have
$$P(R\o 1)\in R\o \varepsilon_t(A) \text{\quad and\quad} (R\o 1)P\in R\o \varepsilon_t(A),$$
$$Q(R\o 1)\in R\o \varepsilon_s(A) \text{\quad and\quad } (R\o 1)U\in R\o \varepsilon_s(A).$$
\end{proposition}
\begin{proof}
The results for $P$ have been explained in part (ii) of Remark \ref{important}. We have
$$(rt\o 1)V_1=(r\o 1)(t\o 1)V_1=\sum rt_{(0)}\o \varepsilon_s(r_{(1)}).$$
Since $W_2$ is regular, the above element belongs to $R\o \varepsilon_s(A)$. Because $R^2=R$, we get
$$(r\o 1)V_1=\sum r_{(0)}\o \varepsilon_s(r_{(1)})\in R\o \varepsilon_s(A).$$
This means $Q(R\o 1)\in R\o \varepsilon_s(A)$. Similar for $U$.
\end{proof}
 
Obviously, if $R$ has an identity, then we get $P\in R\o \varepsilon_t(A)$ and $Q, U\in R\o \varepsilon_s(A)$. Let $r=1$, then the item (4) of Proposition \ref{second} can be expressed as $P=(id\o \varepsilon_t)P$. 

In the end of this section, we consider the basic examples. The examples of groupoid algebra and its dual will discussed in Section \ref{groupoid-example}.

\begin{example}
Let $(A, \Delta, E)$ be a regular (weak) multiplier Hopf algebra. Then $A$ can be viewed as a regular $A$-comodule algebra with the coaction map $\Delta$. Note that our definition does not include the example of  non-regular (weak) multiplier Hopf algebra.
\end{example}

\begin{example}
Assume that $R$ is an $A$-comodule algebra with the counit property. If $A$ is a  multiplier Hopf algebra, then $R$ is a comodule algebra in the sense of Definition 2.1 in \cite{van1999galois}.  
\end{example}

\begin{example}
Assume that $R$ is an $A$-comodule algebra.	Let $A$ be a regular weak multiplier Hopf algebra with identity (now $A$ is a weak Hopf algebra), we get a comodule algebra in the sense of  Definition 2.1 in \cite{bohm2000doi} or Proposition 4.10 in \cite{caenepeel2000modules}. Now the comodule algebra is regular. 
\end{example}

\section{Coactions of weak algebraic quantum groups}
\subsection{A duality theorem for actions and coactions}
Let $\left\langle A, B \right\rangle$ be a weak multiplier Hopf algebra pairing, then we have four unital nondegenerate actions 
\begin{align*}
    A\rhd B: a\rhd b=\sum b_{(1)}\left\langle a, b_{(2)} \right\rangle,\\
    B\rhd A: b\rhd a=\sum a_{(1)}\left\langle a_{(2)}, b \right\rangle,\\
    A\lhd B: a\lhd b=\sum \left\langle a_{(1)}, b \right\rangle a_{(2)},\\
    B\lhd A: b\lhd a=\sum \left\langle a, b_{(1)} \right\rangle b_{(2)}.
\end{align*}
The four actions give us four module algebras. The pairing can be uniquely extended in the following way,
$$\left\langle m, b\lhd a \right\rangle= \left\langle am, b \right\rangle \text{\quad and \quad} 
\left\langle m, a\rhd b \right\rangle= \left\langle ma, b \right\rangle,$$
$$\left\langle a\lhd b, m \right\rangle= \left\langle a, bm \right\rangle \text{\quad and \quad}
\left\langle b\rhd a, m \right\rangle= \left\langle a, mb \right\rangle.$$
Then $\left\langle a, 1 \right\rangle$ and $\left\langle 1, b \right\rangle$ can be defined, and we have
$$\left\langle a, 1 \right\rangle=\varepsilon_A(a) \text{\quad and \quad} \left\langle 1, b \right\rangle=\varepsilon_B(b).$$
But the pairing can not be extended to $M(A)\times M(B)$ generally. For more information about the pairing of WMHA we refer to Section 3 of \cite{zhou2024algebraic} or \cite{tvw2022pairing}.

Let $\left\langle A, B \right\rangle$ be a WMHA pairing. Assume that $(R,\Gamma)$ is a right $A$-comodule algebra. Note that $(id\o \left\langle \cdot, b \right\rangle)\Gamma(r)$ is a multiplier in the natural way. Since the action $A\rhd B$ is unital, we denote $b=\sum a'_i\rhd b'_i$, then 
\begin{equation}\label{eq1}
(id\o \left\langle \cdot, b \right\rangle)\Gamma(r)=\sum (id\o \left\langle \cdot, b'_i \right\rangle)(\Gamma(r)(1\o a'_i))\in R.
\end{equation}
So we can define the action of $B$ on $R$ by 
\begin{equation}\label{co-action}
b\rhd r=(id\o \left\langle \cdot, b \right\rangle)\Gamma(r).
\end{equation}
And we will use the Sweedler notation $b \rhd r=\sum r_{(0)}\left\langle r_{(1)}, b\right\rangle.$

\begin{proposition}
As multipliers, we have $\Gamma(b\rhd r)=(id\o b\rhd \cdot)\Gamma(r)$.
\end{proposition}
\begin{proof}
For any $t\in R$, we have
$$\Gamma(b\rhd r)(t\o 1)=\sum r_{(0)}t\o r_{(1)}\left\langle r_{(2)}, b\right\rangle=\sum r_{(0)}t\o b\rhd r_{(1)}.$$
Cancel $t$, then we get the equation.
\end{proof}

\begin{proposition}
If $\Gamma$ is left full, then the equation \eqref{co-action} gives a unital action. Moreover, if $\Gamma$ is right full, then the action is faithful.
\end{proposition}
\begin{proof}
	For any $b, b'\in B, r\in R$, we have
	\begin{align*}
	b\rhd (b'\rhd r) 
	&= (id\o \left\langle \cdot, b \right\rangle)\Gamma(b'\rhd r)\\
	&=(id\o \left\langle \cdot, b \right\rangle)(id\o b'\rhd \cdot)\Gamma(r)\\
	&=(id\o \left\langle b'\rhd\cdot, b \right\rangle) \Gamma(r)\\
	&=  (id\o \left\langle \cdot, bb' \right\rangle) \Gamma(r)\\
	&=(bb')\rhd r.
	\end{align*}
	
Next let $\omega$ be a non-zero linear functional on $R$ and assume that it is 0 on $B\rhd R$. For any $r\in R, a, b\in A$, we have $\sum \left\langle \omega(r_{(0)})r_{(1)}a, b \right\rangle=0$. So $\sum \omega(r_{(0)})r_{(1)}a=0$ for all $a$. This means that $(\omega\o id)\Gamma(r)=0$ in $M(R)$ for all $r$. Because $\Gamma$ is left full, we get $\omega=0$. So we have a unital action.

Assume that $b\rhd r=0$ for all $r\in R$. Then $\sum r_{(0)}t\left\langle r_{(1)}, b\right\rangle=0$ for all $r, t\in R.$ Since $\Gamma$ is right full, we have $\left\langle a, b\right\rangle=0$ for all $a\in A$. Because the pairing is non-degenerate, we get $b=0$. 
\end{proof}

\begin{proposition}\label{main-1}
With the notation used above. If $R$ is a left-full right $A$-comodule algebra, then $R$ is a left $B$-module algebra with the action defined by equation \eqref{co-action}.
\end{proposition}
\begin{proof}	
Note the expression $\Delta(b)\rhd (r\o r')$ makes sense since we have a unital action. For any $u\in R, r,r'\in R$, we have
\begin{align*}
	u(b\rhd (rr') )&=\sum ur_{(0)}r'_{(0)} \left\langle r_{(1)}r'_{(1)}, b \right\rangle\\
	&=\sum ur_{(0)}r'_{(0)}\left\langle r_{(1)}, b_{(1)} \right\rangle \left\langle r'_{(1)}, b_{(2)} \right\rangle\\
	&=\sum u(b_{(1)}\rhd r)(b_{(2)}\rhd r').
\end{align*}
In the second equality, $r_{(1)}$ and $r'_{(1)}$ are covered by $u$, then $b_{(1)}$ and $b_{(2)}$ are well covered. By the non-degeneracy, we can cancel $u$. So $R$ is a left $B$-module algebra.  
\end{proof}

Assume that $R$ is a left $B$-module algebra. We want to show that the equation \eqref{co-action} gives us a right $A$-comodule algebra structure. This is more complicated and we need the help of integrals. Therefore, we now consider regular weak multiplier Hopf algebras with enough integrals, i.e., weak algebraic quantum groups.

Let $A$ be a regular WMHA. A non-zero linear functional $\varphi: A\mapsto \mathbb{C}$ is called \textit{left integral} if $(id\o \varphi)\Delta(a)\in A_t$ for all $a\in A$. 
The set of left integrals of $A$ is denoted by $\int_l$. The set $\int_l$ is called a \textit{faithful set of integrals} if $\varphi(xa)=0 (\text{or }\varphi(ax)=0)$ for all $\varphi\in \int_l$ and $a\in A$, then $x=0$. Similarly we can defined the right integrals and the set of right integrals of $A$ is denoted by $\int_r$. A \textit{weak algebraic quantum group} is a a regular WMHA with a faithful set of integrals. The dual weak algebraic quantum group $\widehat{A}$ is the linear span of functionals of one of the following four types
$$\varphi(\cdot a), \, \varphi(a\cdot), \, \psi(\cdot a), \, \psi(a\cdot),$$  
for any $\varphi\in \int_l, \psi\in \int_r$ and $a\in A$. Recall that for any $a, b\in A$ and $\varphi,\psi\in A'$, the non-zero linear functional $\varphi$ is left invariant if and only if
\begin{equation}\label{integral-left}
S((id\o \varphi)(\Delta(a)(1\o b)))=(id\o \varphi)((1\o a)\Delta(b)).
\end{equation}
Similarly, $\psi$ is right invariant if and only if
\begin{equation}\label{integral-right}
S((\psi\o id)((a\o 1)\Delta(b)))=(\psi\o id)(\Delta(a)(b\o 1)).
\end{equation}
The counit on $\widehat{A}$ is given by $\varepsilon(\omega)=\omega(1)$ for any $\omega\in \widehat{A}$.

Let $A$ be a weak algebraic quantum group. $R$ is a left $A$-module algebra.
For any $r\in R$ and $ \varphi(\cdot a), \psi(\cdot b)\in \widehat{A}$, here $\varphi$ and $\psi$ are the  left and right  integrals on $A$, respectively. It is possible to define  $\Gamma: R\rightarrow M(R\o \widehat{A})$ by 
\begin{equation}\label{gamma-left}
\Gamma(r)(1\o \varphi(\cdot a))=\sum S^{-1}(a_{(1)})\rhd r\o \varphi(\cdot a_{(2)}),
\end{equation}
\begin{equation}\label{gamma-right}
(1\o \psi(\cdot b))\Gamma(r)=\sum S(b_{(2)})\rhd r\o \psi(\cdot b_{(1)}).
\end{equation}
Since the action is unital, $a_{(1)}$ and $b_{(2)}$ can be covered by $r$, so the right sides of above equations are elements in $R\o \widehat{A}$.  

\begin{proposition}\label{2-1}
For any $r\in R$, we have $\Gamma(r)\in M(R\o \widehat{A})$.
\end{proposition}
\begin{proof}
Take any $\varphi_1(\cdot a), \varphi_2(\cdot b)\in \widehat{A}, a,b\in A$, we have to show that 
$$\Gamma(r)((1\o \varphi_1(\cdot a))(1\o \varphi_2(\cdot b)))=(\Gamma(r)(1\o \varphi_1(\cdot a)))(1\o \varphi_2(\cdot b)).$$
Recall the product in $\widehat{A}$, we have 
	\begin{equation}\label{varphi1varphi2}
	\varphi_1(\cdot a)\varphi_2(\cdot b)=\varphi_2(\cdot p),\, p=\sum b_{(2)}\varphi_1(S^{-1}(b_{(1)})a).
	\end{equation}
It is equivalent to show that 
	$$\sum S^{-1}(p_{(1)})\varphi_2(cp_{(2)})=\sum S^{-1}(a_{(1)})\varphi_1(c_{(1)}a_{(2)})\varphi_2(c_{(2)}b)$$
for any $c\in A$. Note that the expression is meaningful since $T_1$ and $T_3$ are regular. Let us start from the left side, and we have
	\begin{align*}
	LHS
	&= \sum S^{-1}(b_{(2)})\varphi_2(cb_{(3)})\varphi_1(S^{-1}(b_{(1)})a)\\
	&= \sum c_{(1)}\varphi_2(c_{(2)}b_{(2)})\varphi_1(S^{-1}(b_{(1)})a) \tag*{by equation \eqref{integral-left}}\\
	&=\sum c_{(1)}\varphi_2(c_{(3)}b)\varphi_1(c_{(2)}a) \tag*{by equation \eqref{integral-left}}\\
	&= RHS. \tag*{by equation \eqref{integral-left}}
	\end{align*}
Everything is well covered here. We have shown that $\Gamma(r)$ is a left multiplier. Similarly we can show that $\Gamma(r)$ is a right multiplier. Next we need to show that 
$$((1\o \psi(\cdot a'))\Gamma(r))(1\o \varphi(\cdot a))=(1\o \psi(\cdot a'))(\Gamma(r)(1\o \varphi(\cdot a)).$$
It is equivalent to show that 
$$\sum \varphi(c_{(2)}a)\psi(c_{(1)}a'_{(1)})S(a'_{(2)})=\sum \psi(c_{(1)}a')\varphi(c_{(2)}a_{(2)})S^{-1}(a_{(1)})$$
for any $c\in A$. We have
	\begin{align*}
	\sum \varphi(c_{(2)}a)\psi(c_{(1)}a'_{(1)})S(a'_{(2)})
	&= \sum \varphi(c_{(3)}a)\psi(c_{(1)}a')c_{(2)} \tag*{by equation \eqref{integral-right}}\\
	&= \sum \psi(c_{(1)}a')\varphi(c_{(2)}a_{(2)})S^{-1}(a_{(1)}). \tag*{by equation \eqref{integral-left}}
	\end{align*}
So we prove that $\Gamma(r)\in M(R\o \widehat{A})$. 
\end{proof}

In fact, the equations \eqref{gamma-left} and \eqref{gamma-right} are nothing else but the equation \eqref{co-action}. For instance, take any points $p$ of $A$ and use the Sweedler notation $\Gamma(r)=\sum r_{(0)}\o r_{(1)}$, then equation \eqref{gamma-left} means that
$$\sum r_{(0)} \left\langle p_{(1)}\varphi(p_{(2)}a), r_{(1)} \right\rangle =\sum p_{(1)}\varphi(p_{(2)}a)\rhd r.$$
Since $\{\sum a_{(1)}\varphi(ba_{(2)})|a,b\in A, \varphi\in \int_l\}$ spans $A$, so we get 
$$(id\o \left\langle a, \cdot \right\rangle)\Gamma(r)=\sum r_{(0)} \left\langle a, r_{(1)} \right\rangle = a\rhd r,$$ 
here $r_{(1)}$ is covered by $a$ through the pairing.

\begin{proposition}
$\Gamma$ is left full.
\end{proposition}
\begin{proof}
We have $(id\o \varepsilon)(\Gamma(r)(1\o \varphi(\cdot a)))=\sum S^{-1}(a_{(1)}\varphi(a_{(2)}))\rhd r$. Since $\{(id\o \varphi)\Delta(a)| a\in A,\varphi\in \int_l\}=\varepsilon_s(A)$ and $\varepsilon_t(A)A=A$, then elements of the form $\sum S^{-1}(a_{(1)}\varphi(a_{(2)}))\rhd r$ span $R$. So the left leg of $\Gamma$ is $R$.
\end{proof}

\begin{proposition}
$\Gamma$ is a homomorphism.
\end{proposition}		
\begin{proof}
For any $r, r'\in R, a\in A, \varphi\in \int_l$, we compute
	\begin{align*}
	\Gamma(r)\Gamma(r')(1\o \varphi(\cdot a))
	&=\sum \Gamma(r)(S^{-1}(a_{(1)})\rhd r'\o \varphi(\cdot a_{(2)}))\\
	&=\sum  (S^{-1}(a_{(2)})\rhd r)(S^{-1}(a_{(1)})\rhd r')\o \varphi(\cdot a_{(3)})\\
	&=\sum S^{-1}(a_{(1)})\rhd (rr')\o \varphi(\cdot a_{(2)})\\
	&= \Gamma(rr')(1\o \varphi(\cdot a)).
	\end{align*}
Since the action $\rhd$ is unital, elements are well covered. In the third equality we use the definition of module algebra.
\end{proof}

\begin{proposition}
$\Gamma$ is injective.
\end{proposition}
\begin{proof}
Assume that $\Gamma(r)=0$, then $\Gamma(r)(1\o \varphi (\cdot a))=0$ for all $\varphi(\cdot a)\in \widehat{A}$, then $\sum \varphi(a'a_{(2)})S^{-1}(a_{(1)})\rhd r=0$  for all $a,a'\in A$ and $ \varphi\in \int_l$. Again by the fact that $\{\sum a_{(1)}\varphi(a'a_{(2)})|a,a'\in A, \varphi\in \int_l\}$ spans $A$,  we get $r=0$. So $\Gamma$ is injective.
\end{proof}

\begin{proposition}
$\Gamma$ is coassociative in the sense of equation \eqref{coass-1}.
\end{proposition}	
\begin{proof}
	For any $x, y\in A$, we have
	\begin{align*}
	& (id\o \left\langle x, \cdot \right\rangle\o \left\langle y, \cdot \right\rangle)(\Gamma\o id)(\Gamma(r)(1\o \varphi_2(\cdot b)))(1\o \varphi_1(\cdot a)\o 1)\\
	&=\sum (id\o \left\langle x, \cdot \right\rangle\o \left\langle y, \cdot \right\rangle)(\Gamma\o id)(S^{-1}(b_{(1)})\rhd r \o \varphi_2(\cdot b_{(2)}))(1\o \varphi_1(\cdot a)\o 1)\\
	&=\sum (id\o \left\langle x, \cdot \right\rangle\o \left\langle y, \cdot \right\rangle)\Gamma(S^{-1}(b_{(1)})\rhd r)(1\o \varphi_1(\cdot a))\o \varphi_2(\cdot b_{(2)})\\
	&=\sum S^{-1}(b_{(1)}a_{(1)})\rhd r \varphi_1(x a_{(2)}) \varphi_2(y b_{(2)})\\
	&=\sum (x_{(1)}y_{(1)})\rhd r \varphi_1(x_{(2)} a) \varphi_2(y_{(2)} b). \tag*{by (\ref{integral-left})}
	\end{align*}
	Denote $R_1(\varphi_1(\cdot a)\o \varphi_2(\cdot b))=\sum \varphi_{3i}(\cdot p_i)\o \varphi_{4i}(\cdot q_i)$, then we have
	\begin{equation}\label{varphi-1234}
	\sum \varphi_1(xS(y_{(1)})a)\varphi_2(y_{(2)}b)=\sum \varphi_{3i}(xp_i)\varphi_{4i}(yq_i).
	\end{equation}
	Next we compute
	\begin{align*}
	& (id\o \left\langle x, \cdot \right\rangle\o \left\langle y, \cdot \right\rangle)((id\o \Delta(r))\Gamma(r))(1\o \varphi_1(\cdot a)\o
	\varphi_2(\cdot b))\\
	&= \sum  (id\o \left\langle x, \cdot \right\rangle\o \left\langle y, \cdot \right\rangle) (id\o \Delta)(\Gamma(r)(1\o \varphi_{3i}(\cdot p_i)))(1\o 1\o \varphi_{4i}(\cdot q_i))\\
	&=\sum \sum  (id\o \left\langle x, \cdot \right\rangle\o \left\langle y, \cdot \right\rangle)  S^{-1}(p_{i(1)})\rhd r\o \Delta(\varphi_{3i}(\cdot p_{i(2)}))(1\o \varphi_{4i}(\cdot q_i))\\
	&= \sum \sum S^{-1}(p_{i(1)})\rhd r \varphi_{3i}(xS^{-1}(q_{i(1)})p_{i(2)})\varphi_4(yq_{i(2)})\\
	&= \sum \sum (x_{(1)}y_{(1)})\rhd r\varphi_{3i}(x_{(2)}y_{(2)}p_i)\varphi_4(y_{(3)}q_i) \tag*{by equation \eqref{integral-left}}\\
	&= \sum (x_{(1)}y_{(1)})\rhd r\varphi_1(x_{(2)}\varepsilon_t(y_{(2)})a)\varphi_2(y_{(3)}b) \tag*{by equation \eqref{varphi-1234}}\\
	&=\sum (x_{(1)}E_1y_{(1)})\rhd r\varphi_1(x_{(2)}E_2a)\varphi_2(y_{(2)}b) \\
	&=\sum (x_{(1)}y_{(1)})\rhd r \varphi_1(x_{(2)} a) \varphi_2(y_{(2)} b) 
	\end{align*}
Everything is well covered. We have proved the coassociativity of $\Gamma$.
\end{proof}

We now have a semi-regular injective coaction. Next define $P$ by 
$$P(r\o \varphi(\cdot a))=\sum  \varepsilon_t(S^{-1}(a_{(1)}))\rhd r\o \varphi(\cdot a_{(2)})$$
$$(r\o \psi(\cdot a))P=\sum \varepsilon_s(a_{(2)})\rhd r\o \psi(\cdot a_{(1)}).$$
Since the action is unital, the right sides of above equations are elements in $R\o \widehat{A}$. 

\begin{proposition}
$P\in M(R\o \widehat{A})$.
\end{proposition}
\begin{proof}
For any $x\in A$, we have
	\begin{align*}
	&(id\o \left\langle x , \cdot\right\rangle)P(r\o \varphi_1(\cdot a))(u\o \varphi_2(\cdot b))\\
	&=\sum (\varepsilon_t(S^{-1}(a_{(1)}))\rhd r)u\varphi_1(x_{(1)}a_{(2)})\varphi_2(x_{(2)}b)\\
	&=\sum (\varepsilon_t(x_{(1)})\rhd r)u\varphi_1(x_{(2)}a)\varphi_2(x_{(3)}b) \tag*{by equation \eqref{integral-left}}
	\end{align*}
and
	\begin{align*}
	(id\o \left\langle x , \cdot\right\rangle)P(ru\o \varphi_1(\cdot a)\varphi_2(\cdot b))
	&=(id\o \left\langle x , \cdot\right\rangle)P(ru\o \varphi_2(\cdot p)) \tag*{by equation \eqref{varphi1varphi2}}\\
	&= \sum \varepsilon_t(S^{-1}(b_{(2)}))\rhd (ru)\varphi_1(S^{-1}(b_{(1)})a)\varphi_2(xb_{(3)})\\
	&=\sum \varepsilon_t(x_{(1)})\rhd (ru)\varphi_1(S^{-1}(b_{(1)})a)\varphi_2(x_{(2)}b_{(2)}) \tag*{by equation \eqref{integral-left}}\\
	&= \sum \varepsilon_t(x_{(1)})\rhd(ru)\varphi_1(x_{(2)}a)\varphi_2(x_{(3)}b). \tag*{by equation \eqref{integral-left}}
	\end{align*}
By equation \eqref{xru}, we obtain that $P$ is a left multiplier. Similarly, $P$ is a right multiplier. Next we have
	\begin{align*}
	(id\o \left\langle x , \cdot\right\rangle)((r\o \psi(\cdot a))P)(u\o \varphi(\cdot b))
	&=\sum (\varepsilon_s(a_{(2)})\rhd r)u\psi(x_{(1)}a_{(1)})\varphi(x_{(2)}b)\\
	&=\sum (\varepsilon_s(S^{-1}(a_{(2)}))\rhd r)u\psi(x_{(1)}a)\varphi(x_{(3)}b) \tag*{by equation \eqref{integral-right}}
	\end{align*}
	and
	\begin{align*}
	(id\o \left\langle x , \cdot\right\rangle)(r\o \psi(\cdot a))(P(u\o \varphi(\cdot b)))
	&=\sum r(\varepsilon_t(S^{-1}(b_{(1)}))\rhd u)\psi(x_{(1)}a)\varphi(x_{(2)}b_{(2)})\\
	&=\sum r(\varepsilon_t(x_{(2)})\rhd u)\psi(x_{(1)}a)\varphi(x_{(3)}b). \tag*{by equation \eqref{integral-left}}
	\end{align*}
By equation \eqref{yru} and $S\circ \varepsilon_s=\varepsilon_t\circ S$, we get
	$$((r\o \psi(\cdot a))P)(u\o \varphi(\cdot b))=(r\o \psi(\cdot a))(P(u\o \varphi(\cdot b))).$$
This means that $P$ is a multiplier.
\end{proof}

\begin{proposition}
$P$ is an idempotent.
\end{proposition}

\begin{proof}
\begin{align*}
P(P(r\o \varphi(\cdot a))) 
&=\sum P(\varepsilon_t(S^{-1}(a_{(1)}))\rhd r\o \varphi(\cdot a_{(2)}))\\
&=\sum \varepsilon_t(S^{-1}(a_{(2)}))\varepsilon_t(S^{-1}(a_{(1)}))\rhd r\o \varphi(\cdot a_{(3)}))\\
&=\sum \varepsilon_t(S^{-1}(a_{(1)}))\rhd r\o \varphi(\cdot a_{(2)}).
\end{align*}
Everything can be covered by the unital action.
\end{proof}

\begin{proposition}
$P$ satisfies the Assumption I.
\end{proposition}
\begin{proof}
Take any point $x\in A$, we have
$$(id\o \left\langle x , \cdot\right\rangle)P(r\o \varphi(\cdot a))=\sum\varepsilon_t(x_{(1)}\varphi(x_{(2)}a))\rhd r.$$
Since $A\rhd R=R$ and $\varepsilon_t(A)A=A$, the elements on the right side of above the equation span $A$. On the other hand, 
$$(id\o \left\langle x, \cdot\right\rangle)\Gamma(r)(1\o \varphi(\cdot a))=\sum x_{(1)}\varphi(x_{(2)}a)\rhd r.$$
Then we obtain $P(R\o A)=\Gamma(R)(1\o A)$. Similarly for $(R\o A)P=(1\o A)\Gamma(R)$.  
So $P$ satisfies the first assumption.
\end{proof}

In the next proposition, we will use the following notation. Denote 
$\overline{W_1}(r\o \varphi(\cdot a))=\sum S(p_i)\rhd r\o \varphi'(\cdot q_i)$, $p_i,q_i\in A$ and $\varphi, \varphi'\in \int_l$. For any $x\in A$, we have 
\begin{equation}\label{2-2}
\sum S(p_i)\varphi'(xq_i)=\sum S(x_{(1)})\varphi(x_{(2)}a).
\end{equation}
For any $\varphi_1, \varphi_2\in \int_l$, denote $E(\varphi_1(\cdot a)\o \varphi_2(\cdot b))=\sum \varphi_{1i}(\cdot u_i)\o \varphi_{2i}(\cdot v_i)$ with $u_i,v_i\in A$ and $\varphi_{1i},\varphi_{2i}\in \int_l$. Then for any $x\in A$ we have
\begin{equation}\label{2-3}
\sum \varphi_{1i}(xu_i)\varphi_{2i}(yv_i)=\sum \varphi_1(x\varepsilon_t(y_{(1)})a)\varphi_2(y_{(2)}b).
\end{equation}

\begin{proposition}
$(\Gamma\o id)(P)=(1\o E)(P\o 1)$.
\end{proposition}
\begin{proof}
	\begin{align*}
	&(\Gamma\o id)(P)(r\o \varphi_1(\cdot a)\o \varphi_2(\cdot b))\\
	&=(\Gamma\o id)(P)(P(r\o \varphi_1(\cdot a))\o \varphi_2(\cdot b))\\
	&=\sum(\Gamma\o id)(P)(W_1(S(p_i)\rhd r\o \varphi'_1(\cdot q_i))\o \varphi_2(\cdot b))\\
	&=\sum (\Gamma\o id)(\varepsilon_t(S^{-1}(b_{(1)}))S(p_i)\rhd r\o \varphi_2(\cdot b_{(2)}))(1\o \varphi'_1(\cdot q_i)\o 1)\\
	&=\sum S^{-1}(q_{i(1)})\varepsilon_t(S^{-1}(b_{(1)}))S(p_i)\rhd r\o \varphi'_1(\cdot q_{i(2)})\o \varphi_2(\cdot b_{(2)}).
	\end{align*}
Take any $x\o y\in A\o A$, we get 
	\begin{align*}
	&\sum S^{-1}(q_{i(1)})\varepsilon_t(S^{-1}(b_{(1)}))S(p_i) \varphi'_1(x q_{i(2)}) \varphi_2(y
	b_{(2)})\\
	&=\sum x_{(1)}\varepsilon_t(y_{(1)})S(p_i)\varphi'_1(x_{(2)} q_i) \varphi_2(y_{(2)}
	b) \tag*{by equation \eqref{integral-left}}\\
	&=\sum x_{(1)}\varepsilon_t(y_{(1)})S(x_{(2)})\varphi_1(x_{(3)}a)\varphi_2(y_{(2)}b) \tag*{by equation \eqref{2-2}}\\
	&=\sum \varepsilon_t(x_{(1)}y_{(1)})\varphi_1(x_{(2)}a)\varphi_2(y_{(2)}b)
	\end{align*}
Next we have
	\begin{align*}
	&(id\o \left\langle x,\cdot \right\rangle\o \left\langle y,\cdot \right\rangle)(1\o E)(P\o 1)(r\o \varphi_1(\cdot a)\o \varphi_2(\cdot b))\\
	&=\sum\sum \varepsilon_t(S^{-1}(u_{i(1)}))\varphi_{1i}(xu_{i(2)})\varphi_{2i}(yv_i)\\
	&=\sum\sum \varepsilon_t(x_{(1)})\varphi_{1i}(x_{(2)}u_i)\varphi_{2i}(yv_i) \tag*{by equation \eqref{integral-left}}\\
	&=\sum \varepsilon_t(x_{(1)})\varphi_1(x_{(2)}\varepsilon_t(y_{(1)})a)\varphi_2(y_{(2)}b) \tag*{by equation \eqref{2-3}}\\
	&=\sum \varepsilon_t(x_{(1)}S^{-1}(\varepsilon_t(y_{(1)})))\varphi_1(x_{(2)}a)\varphi_2(y_{(2)}b)\tag*{by equation \eqref{Ey}}\\
	&=\sum \varepsilon_t(x_{(1)}y_{(1)})\varphi_1(x_{(2)}a)\varphi_2(y_{(2)}b)
	\end{align*}
So we get $(\Gamma\o id)(P)=(1\o E)(P\o 1)$.  Then we finish the proof.
\end{proof}
So far, we have proved the following result.

\begin{proposition}\label{main-2}
Let $A$ be a weak algebraic quantum group. If $R$ is a left $A$-module algebra, then $R$ is a left-full right $\widehat{A}$-comodule algebra.
\end{proposition}

By Proposition \ref{main-1} and \ref{main-2}, we obtain the main theorem in this section.

\begin{theorem}\label{action-coaction}
Let $A$ be a weak algebraic quantum group. Then $R$ is an $A$-module algebra if and only if $R$ is a left full $\widehat{A}$-comodule algebra.
\end{theorem}

So we generalize \cite[Theorem 3.3]{van1999galois} to the weak algebraic quantum group setting. It was shown in \cite{zhou2024algebraic} that a vector space $R$ is a unital left $D(A)$-module if and only if $R$ is a Yetter–Drinfeld $A$-module. Recall that an algebra $R$ is a Yetter–Drinfeld module algebra if $R$ is a left $A$-module algebra and a left $A$-comodule algebra such that $R$ is a Yetter–Drinfeld module. Now we have the following result.

\begin{corollary}
Let $A$ be a weak algebraic quantum group. Then $R$ is a Yetter–Drinfeld module algebra if and only if $R$ is a $D(A)$-module algebra.
\end{corollary}

\subsection{The Blattner-Montgomery type theorem for weak algebraic quantum groups}

Let $\left\langle A, B\right\rangle$ be a weak algebraic quantum group pairing and $R$  a $B$-comodule algebra. We have shown that there is an action of $A$ comes from the coaction of $B$. With this action, we can form the smash product $R\# A$. Recall from Section 6 of \cite{zhou2017duality}, we have two maps $\Phi$ and $\Phi_1$, given by 
$$\Phi: R\o A\rightarrow R\# A, \quad r\o a\mapsto \sum a_{(1)}\rhd r\# a_{(2)},$$
$$\Phi_1: R\# A\rightarrow R\o A, \quad r\#a\mapsto \sum S^{-1}(a_{(1)})\rhd r\o a_{(2)}.$$
And we have 
$$\Phi\Phi_1\Phi=\Phi, \quad \quad \Phi_1\Phi\Phi_1=\Phi_1.$$
Denote $$R\bar{\o}A=\{r\o a\in R\o A\mid r\o a=\sum S^{-1}(E_1)\rhd r\o aE_2\},$$
then $\Phi$ is bijective on $R\bar{\o}A$ with the inverse $\Phi_1$. We also know that $R\# A$ is a left $B$-module algebra with the action $b\rhd (r\# a)=r\# (b\rhd a)$. So we have the smash product $(R\# A)\# B$. And $R\# A$ is a left faithful $(R\# A)\# B$-module with action $\rhd_1$ defined by
$$(r\# a)\# b)\rhd_1 (r'\# a')=(r\# a)(r'\# b\rhd a').$$
By \cite[Proposition 6.3]{zhou2017duality}, $R\bar{\o}A$ is a faithful $(R\# A)\# B$-module with the action $\rhd_2$ defined by 
$$((r\# a)\# b)\rhd_2(r'\bar{\o} a')=\Phi_1((r\# a)\# b)\rhd_1 \Phi(r'\# a')).$$
Since $A$ is a left faithful $A\# B$-module with action $(a\# b)\rhd a'=a(b\rhd a')$, $R\bar{\o}A$ is a left faithful $R\o (A\# B)$-module with the action $\rhd_3$ defined by
$$(r\o (a\# b))\rhd_3(r'\bar{\o} a')=rr'\bar{\o} ((a\# b)\rhd a').$$

\begin{theorem}\label{bidual}
    Let $\left\langle A, B\right\rangle$ be a weak algebraic quantum group pairing and $R$ a $B$-comodule algebra. If the right action of $B$ on $A$ can be expressed as a left action of  a multiplier of $R\# A$ as acting on $A$. Then
    $$(R\# A)\# B\cong R\o (A\# B).$$
\end{theorem}
\begin{proof}
    We need to show that the operator defined by the action of $(r\#a) \#b$ via $\rhd_2$ is equal to the operator defined by some element of $R \otimes (A \# B)$ acting via $\rhd_3$. By definition, we have
    $$((r\# a)\# b)\rhd_2(r'\bar{\o} a')=\sum (S^{-1}(a_{(1)}a'_{(1)})\rhd r)r'\bar{\o} a_{(2)}(b\rhd a'_{(2)}).$$
    Since $\Phi_1$ is bijective on $R\bar{\o} A$, the operator is of the form 
    $$r'\bar{\o} a' \mapsto \sum (S^{-1}(a'_{(1)})\rhd r)r'\bar{\o} a(b\rhd a'_{(2)}).$$
    This is equivalent with
    $$r'\bar{\o} a' \mapsto \sum r_{(0)}r'\bar{\o} (a\# b)\rhd a'\lhd S^{-1}(r_{(1)}).$$
    By the assumption in the theorem, the above operator can be denoted by
    $$r'\bar{\o} a' \mapsto \sum ur'\bar{\o} (v\# w)\rhd a'$$
    with certain $u\in R, v\# w\in A\# B$. Since the two modules are faithful, we get an isomorphism. 
\end{proof}
This theorem generalizes the Blattner-Montgomery duality theorem for Hopf algebras (see \cite[Theorem 9.4.9]{montgomery1993hopf}) to the setting of weak multiplier Hopf algebras.  The original version of this duality was proved by M. Takesaki for actions of locally compact groups on von Neumann algebras (\cite{takesaki1973duality}). Specifically, let $A$ be a weak algebraic quantum group and consider the pairing $\langle A, \hat{A} \rangle$. By Theorem \ref{action-coaction}, every action of $A$ arises from a coaction of $\hat{A}$. Consequently, our result unifies and extends the known duality theorems for multiplier Hopf algebras \cite[Theorems 7.6, 7.8]{drabant1999actions}, weak Hopf algebras \cite[Theorem 3.3]{nikshych2000duality} and weak multiplier Hopf algebras \cite[Theorem 6.1]{zhou2017duality}. Our method used here comes from \cite{drabant1999actions}. In the Example \ref{groupoiddual}, we will find an explicit expression for the isomorphism.

\section{Morita context of weak algebraic quantum group}\label{moritacontext}

\subsection{The fixed points}

Let $A$ be a weak algebraic quantum group and  $R$ a left $A$-module algebra with the action $\rhd$. By Theorem \ref{action-coaction} we get a right coaction $\Gamma: R\rightarrow M(R\o \hat{A})$ associated with $\rhd$. 

\begin{proposition}\label{coinvariant-equ}
For any $x\in M(R), r\in R, a\in A$, the following equations are equivalent:  
            \begin{enumerate}[(i)]
                \item $a\rhd (rx)=(a\rhd r)x$,    
                \item $(a\rhd x)r=\varepsilon_t(a)\rhd (xr)$,
                \item $r(a\rhd x)=(\varepsilon'_s(a)\rhd r)x$,
                \item $\Gamma(rx)=\Gamma(r)(x\o 1)$,
                \item $\Gamma(x)=P(x\o 1)$,
            \end{enumerate}
            \begin{enumerate}[(a)]
                \item $a\rhd (xr)=x(a\rhd r)$,
                \item $r(a\rhd x)=\varepsilon'_s(a)\rhd (rx)$,
                \item $(a\rhd x)r=x(\varepsilon_t(a)\rhd r)$,
                \item $\Gamma(xr)=(x\o 1)\Gamma(r),$
                \item $\Gamma(x)=(x\o 1)P.$ 
            \end{enumerate}
\end{proposition}
\begin{proof}
Step 1: $(i)\implies (ii)$ and $(i)\implies (iii)$. For any $t\in R$, we have 
    \begin{align*}
        t(a\rhd x)r &=\sum (a_{(2)}\rhd ((S^{-1}(a_{(1)})\rhd t)x))r\\
                    &\stackrel{(i)}{=}\sum (\varepsilon'_s(a)\rhd t)xr\\
                    &\stackrel{\eqref{yru}}{=}\sum t(\varepsilon_t(a)\rhd (xr)).
    \end{align*}
We can cancel $r$ in the third equality, then we obtain (iii). By the last equality, we obtain (ii). 

$(iii)\iff (ii)$. For any $t\in R$, we have
$$t(a\rhd x)r\stackrel{(iii)}{=} (\varepsilon_s'(a)\rhd t)xr\stackrel{\eqref{yru}}{=}t(\varepsilon_t(a)\rhd (xr)).$$
Cancel $t$, then we obtain (ii). Conversely, we use formula \eqref{xru} and the proof is similar.

$(iii)\implies (i)$. Note that the expression $\sum (a_{(1)}\rhd r)(a_{(2)}\rhd x)$ is meaningful since $a_{(1)}$ can be covered by the unital action. We have
\begin{align*}
    \sum (a_{(1)}\rhd r)(a_{(2)}\rhd x)t &= \sum (a_{(1)}\rhd r)(a_{(2)}\rhd (x(S(a_{(3)})\rhd t)))\\
    &= \sum a_{(1)}\rhd (rx(S(a_{(2)})\rhd t))\\
    &= (a\rhd (rx))t.
\end{align*}
Now we are free to use the formula $a\rhd (rx)=\sum (a_{(1)}\rhd r)(a_{(2)}\rhd x)$. Next we get
\begin{align*}
    a\rhd (rx)&=\sum (a_{(1)}\rhd r)(a_{(2)}\rhd x)\\
    &\stackrel{(iii)}{=}\sum (\varepsilon'_s(a_{(2)})\rhd (a_{(1)}\rhd r))x\\
    &=(a\rhd r)x.
\end{align*}

$(iv)\implies (i)$. 
    \begin{align*}
        a\rhd (rx) &= (id\o \left\langle  \cdot, a \right\rangle)\Gamma(rx)\\
                   &= (id\o \left\langle  \cdot, a \right\rangle)\Gamma(r)(x\o 1)\\
                   &= \sum r_{(0)}x \left\langle  r_{(1)}, a \right\rangle\\
                   &= (a\rhd r)x.
    \end{align*}

$(i)\implies (iv)$. For any $\varphi(\cdot a)\in \widehat{A}$, we have
    \begin{align*}
        \Gamma(rx)(1\o \varphi(\cdot a))&=\sum S^{-1}(a_{(1)})\rhd (rx)\o \varphi(\cdot a_{(2)})\\
         &\stackrel{(i)}{=} \sum (S^{-1}(a_{(1)})\rhd r)x\o \varphi(\cdot a_{(2)})\\
         &=\Gamma(r)(x\o 1)(1\o \varphi(\cdot a)).
    \end{align*}

$(iv)\implies (v)$. For any element $\hat{a}\in \hat{A}$, we have
    $$(1\o \hat{a})\Gamma(rx)=(1\o \hat{a})\Gamma(r)(x\o 1).$$
Recall that $(R\o \hat{A})P=(1\o \hat{A})\Gamma(R)$ and $P\Gamma(r)=\Gamma(r)$, then we obtain formula $(v)$. From (v) to (iv), it is easy. We have shown that the equations  (i)-(v) are equivalent.

Step 2: The equivalence among the last five equations can be proved in a similar way.

Step 3: $(1)\implies (b)$. 
   \begin{align*}
      r(a\rhd x) &\stackrel{\eqref{left-r}}{=} \sum a_{(2)}\rhd ((S^{-1}(a_{(1)})\rhd r)x)\\
                 &\stackrel{(i)}{=}\sum a_{(2)}\rhd (S^{-1}(a_{(1)})\rhd (rx))\\
                 &= \varepsilon_s'(a)\rhd (rx).
   \end{align*}
In a similar way we can show that $(a)\implies (ii)$. Then we finish the proof.
\end{proof}

\begin{definition}
Let $A$ be a WMHA and $R$ a nondegenerate algebra.
\begin{enumerate}[(1)] 
    \item  Assume that $R$ is an $A$-module algebra. The fixed point algebra $R^A$ is defined as the subalgebra of $M(R)$ consisting of all elements $x$ that satisfy one of the conditions (i), (ii), (a), or (c) in Proposition \ref{coinvariant-equ}. 
    \item  Assume that $(R, \Gamma, P)$ an $A$-comodule algebra. An element $r \in M(R)$ is a coinvariant if it satisfies one of the conditions (iv), (v), (d), or (e) in Proposition \ref{coinvariant-equ}.  The set of all coinvariants of $M(R)$ form a unital subalgebra of $M(R)$, denoted by $R^{\text{co} A}$. 
\end{enumerate}
\end{definition}

Note that equations $(iii)$ and $(b)$ of Proposition \ref{coinvariant-equ} are only meaningful when $A$ is regular. As a consequence of Proposition \ref{coinvariant-equ}, we obtain the following result.

\begin{corollary}
Let $A$ be a weak algebraic quantum group and $R$ a left $A$-module algebra. Then $R^{\text{co} A}=R^{\hat{A}}$.
\end{corollary}

\begin{definition}
Let $A$ be a weak algebraic quantum group and $(R, \Gamma, P)$ an regular $A$-comodule algebra. $\Gamma$ is called a Galois coaction if the map
    $$ \nu: R\o_{R^{\textbf{co} A}} R\rightarrow (R\o A)P, r\o r'\mapsto (r\o 1)\Gamma(r')$$
is bijective.
\end{definition}
Since $\Gamma$ is regular, we have $(r\o 1)\Gamma(r')=(r\o 1)\Gamma(r')P\in (R\o A)P$. For any $x\in R^{\textbf{co} A}$, we have 
 $$\nu (rx\o r')=(rx\o 1)\Gamma (r')=(r\o 1)(x\o 1)P\Gamma(r')=(r\o 1)\Gamma(xr')=\nu (r\o xr').$$
In the third equality we use (e) of Proposition \ref{coinvariant-equ}. So $\nu$ is well defined.

\begin{proposition}
Let $\Gamma$ be a Galois coaction of $A$ on $R$. If $R$ has unit $1$, then $\Delta(A)\in A\o A$.
\end{proposition}
\begin{proof}
Since $R$ has unit, we have $\Gamma(r)=(1\o 1)\Gamma(r)\in R\o A$ for all $r\in R$. For any $a\in A$, by the bijectivity of $\nu$, there exists an element $\sum u_{(i)}\o v_{(i)}\in R\o_{R^{\textbf{co} A}} R$ such that 
    $$\nu \big(\sum_i u_{(i)}\o v_{(i)}\big)=\sum_i (u_{(i)}\o 1)\Gamma(v_{(i)})=1\o a.$$
Applying $id\o \Delta$, we get 
\begin{align*}
    1\o \Delta (a) &= \sum_i (u_{(i)}\o 1\o 1)(id\o \Delta)\Gamma(v_{(i)})\\
                   &= \sum_i (u_{(i)}\o 1\o 1)(\Gamma\o id)\Gamma(v_{(i)})\in R\o A\o A.
\end{align*}
So $\Delta(a)\in A\o A$ for any $a\in A$.
\end{proof}

Now a natural question arises: given a WMHA $(A, \Delta)$, if $\Delta(A)\subseteq A\o A$, do we have $1\in A$? It is true if $A$ is a multiplier Hopf algebra because we have $\varepsilon(a)1=\sum a_{(1)}S(a_{(2)})\in A$, then $A$ is a Hopf algebra. If $A$ is a multiplier bialgebra in the sense of \cite[Theorem 2.11]{BGL}, then we get $1\in A$ and now $A$ is the usual bialgebra. If $A$ is a WMHA, then we can obtain $\varepsilon_s(A)\in A$ and $\varepsilon_t(A)\in A$. Even in the regular case, we are not able to show that $1\in A$. Note that for a non-regular WMHA with a unit, we do not know whether it is a weak Hopf algebra or not. This is open.

\subsection{Morita context} 

Let $A$ be a weak algebraic quantum group. Assume that there is a faithful left integral $\varphi$ on $A$. We recall the following results from \cite[Theorem 2.8]{van2020weak} and \cite[Proposition 2.9]{van2020weak}.

\begin{lemma}\label{faithfulintegral}
\begin{enumerate}[(1)]
    \item Assume that $\varphi_1$ is a left integral, then there is an element $y\in A_s$ so that $\varphi_1(a)=\varphi(ay)$ for all $a\in A$.
    \item Assume that $\psi$ is a right integral, then there is an element $\delta\in M(A)$ such that $\psi(a)=\varphi(a\delta)$ for all $a\in A$. If $\psi$ is faithful, then $\delta$ is invertible.
\end{enumerate}
\end{lemma}

\begin{lemma}
Assume that $A$ is a weak algebraic quantum group and $R$  an idempotent algebra. Let $(R,\Gamma,P)$ be a right $A$-comodule algebra. 
\begin{enumerate}[(1)]
    \item $R$ is a $(R\# \widehat{A}, R^{\textbf{co} A})$-bimodule with the following module structures:
    \begin{align*}
        & \text{left:\quad } (r\# \varphi(\cdot a))\rhd r'=r(\varphi(\cdot a)\rhd r'),\\
        & \text{right: \quad}  r\lhd x=rx,
    \end{align*}
    for all $r,r'\in R, x\in R^{\textbf{co} A}, \varphi(\cdot a)\in \widehat{A}$.
    \item  $R$ is a $(R^{\textbf{co} A}, R\# \widehat{A})$-bimodule with the following module structures:
    \begin{align*}
        & \text{right: \quad} r\lhd (r'\#\varphi(\cdot a))= S^{-1}(\varphi(\cdot a))\rhd (rr'),\\
        & \text{right: \quad} x\rhd r=xr,
    \end{align*}
    for all $r,r'\in R, x\in R^{\text{co} A}, \varphi(\cdot a)\in \widehat{A}$.
\end{enumerate}
\end{lemma}
\begin{proof}
By Theorem \ref{action-coaction}, $R$ is a left $\hat{A}$-module algebra with action $\rhd$. Then we can form the smash product $R\# \widehat{A}$. It was proved in \cite{zhou2017duality}  that $R$ is a unital left $R\# \widehat{A}$-module. Next we show that $R$ is a unital right $R\# \widehat{A}$-module. We have
   \begin{align*}
       (r\lhd (r'\# \varphi(\cdot a)))\lhd (r''\# \omega(\cdot b))
       &= (S^{-1}(\varphi(\cdot a))\rhd (rr'))\lhd (r''\# \omega(\cdot b))\\
       &=S^{-1}(\omega(\cdot b))\rhd ((S^{-1}(\varphi(\cdot a))\rhd (rr'))r'')
   \end{align*}
and   
    \begin{align*}
       r\lhd ((r'\# \varphi(\cdot a))(r''\# \omega(\cdot b)))
       &= \sum r\lhd (r'(\varphi(\cdot a)_{(1)}\rhd r'')\# \varphi(\cdot a)_{(2)}\omega(\cdot b))\\
       &=\sum S^{-1}(\omega(\cdot b))\rhd (S^{-1}(\varphi(\cdot a)_{(2)})\rhd (rr'(\varphi(\cdot a)_{(1)}\rhd r'')))\\
       &=S^{-1}(\omega(\cdot b))\rhd ((S^{-1}(\varphi(\cdot a))\rhd (rr'))r'').
   \end{align*}
In the last equality we use 
   $$\sum S^{-1}(\hat{a}_{(2)})\rhd (r(\hat{a}_{(1)}\rhd r'))=\sum (S^{-1}(\hat{a}_{(2)})\rhd r)(\varepsilon'_t(\hat{a}_{(1)})\rhd r') \stackrel{\eqref{yru}}{=}(S^{-1}(\hat{a})\rhd r)r'$$
for any $\hat{a}\in \hat{A}$. Since $R$ is idempotent, the right $R\# \widehat{A}$-module is unital. The remaining steps are routine.
\end{proof}

\begin{lemma}
Assume that there exits a two-sided integral  $\varphi$ on $A$. Define a map $[\cdot,\cdot]: R\o_{R^{\text{co} A}} R \rightarrow R\# \hat{A}$ by 
    $$[r, r']=\sum rr'_{(0)}\# \varphi(r'_{(1)}\cdot)$$
for all $r,r'\in R$. Then $[\cdot , \cdot]$ is a $(R\# \widehat{A},R\# \widehat{A})$-bimodule map.
\end{lemma}
\begin{proof}
    (1) First, we have to show that $[\cdot, \cdot]$ is well defined. Since $W_2$ is regular,  we have $\sum rr'_{(0)}\# \varphi(r'_{(1)}\cdot)\in R\o \hat{A}$.  For any $r,r'\in R, x\in R^{\text{co} A}$,
    $$[rx, r']=\sum rxr'_{(0)}\# \varphi(r'_{(1)}\cdot )=\sum r(xr')_{(0)}\# \varphi ((xr')_{(1)}\cdot)=[r, xr'].$$
    In the second equality we use (d) of Proposition \ref{coinvariant-equ}. 

(2) Next we view $R\#\hat{A}$ as a subset of $\text{Hom}(A, R)$ and we prove that $[(r\o \hat{a})\rhd r', r'']$ and $(r\o \hat{a})[r',r'']$ take the same value in the points $c$ of $A$. Then we have to show 
     $$r\sum (\hat{a}\rhd r')r''_{(0)}\varphi(r''_{(1)}c)=\sum r(\hat{a}_{(1)}\rhd (r'r''_{(0)}))\hat{a}(c_{(1)})\varphi(r''_{(1)}c_{(2)}).$$
We have
    \begin{align*}
        RHS &=\sum r(\hat{a}_{(1)}\rhd r')(\hat{a}_{(2)}\rhd r''_{(0)})\hat{a}_{(3)}(c_{(1)})\varphi(r''_{(1)}c_{(2)})\\
         &=\sum r(\hat{a}_{(1)}\rhd r')r''_{(0)}\hat{a}_{(2)}(r''_{(1)}c_{(1)})\varphi(r''_{(2)}c_{(2)})\\
         &=\sum r r'_{(0)}r''_{(0)}\left\langle r'_{(1)}, \hat{a}_{(1)} \right\rangle 
         \left\langle \varepsilon_t(r''_{(1)}), \hat{a}_{(2)} \right\rangle \varphi(r''_{(2)}c)\\
         &= \sum rr'_{(0)}r''_{(0)}\varepsilon(r'_{(1)}r''_{(1)}) \left\langle r'_{(2)}, \hat{a}  \right\rangle \varphi(r''_{(2)}c)\\
         &= \sum  rr'_{0}r''_{(0)} \left\langle r'_{(1)}, \hat{a}  \right\rangle \varphi(r''_{(1)}c)\\
         &= LHS.
    \end{align*}
In the third equality, we use the left invariance of $\varphi$. In the fourth equality we use $\left\langle a, 1 \right\rangle=\varepsilon(a)$ and $\left\langle \varepsilon_t(b), \hat{a} \right\rangle=\left\langle b,\varepsilon_t(\hat{a}) \right\rangle$. Similarly, we have to show that $[r,r'\lhd (r''\# \hat{a})]=[r, r'](r''\# \hat{a})$. It is equivalent to show that
    $$\sum r(S^{-1}(\hat{a})\rhd (r'r''))_{(0)}\# \varphi ((S^{-1}(\hat{a})\rhd (r'r''))_{(1)}\cdot)=\sum rr'_{(0)}(\varphi(r'_{(1)}\cdot)_{(1)}\rhd r'')\# \varphi(r'_{(1)}\cdot)_{(2)}\hat{a}.$$
Take any point $c\in A$, the left side becomes
        $$ \sum r(r'r'')_{(0)}\varphi((r'r'')_{(1)}c)\left\langle (r'r'')_{(2)}, S^{-1}(\hat{a}) \right\rangle,$$
the right side becomes 
        \begin{align*}
        & \sum rr'_{(0)}(\varphi(r'_{(1)}\cdot)_{(1)}\rhd r'') \left\langle c_{(1)}, \varphi(r'_{(1)}\cdot)_{(2)} \right\rangle  \left\langle c_{(2)}, \hat{a} \right\rangle \\   
         &= \sum rr'_{(0)}r''_{(0)} \left\langle r''_{(1)}, \varphi(r'_{(1)}\cdot)_{(1)}\right\rangle \left\langle c_{(1)}, \varphi(r'_{(1)}\cdot)_{(2)} \right\rangle  \left\langle c_{(2)}, \hat{a} \right\rangle\\
         &=\sum rr'_{(0)}r''_{(0)}\varphi(r'_{(1)}r''_{(1)}c_{(1)})\left\langle c_{(2)}, \hat{a} \right\rangle.
        \end{align*}
By the right invariance of $\varphi$, we get the equality.
\end{proof}

\begin{lemma}
Let $\varphi$ be a two-sided integral of $A$. Define a map  $(\cdot, \cdot) : R\o_{R\# \widehat{A}} R\rightarrow R^{\text{co}A}$ by 
    $$ (r, r')=(id\o \varphi)\Gamma(rr')$$
for all $r,r'\in R$. Then $(\cdot, \cdot)$ is a $(R^{\text{co}A}, R^{\text{co}A})$-bimodule map. 
\end{lemma}
\begin{proof}
Denote $(id \o \varphi)\Gamma(r)=x$, then $x$ is a multiplier in the natural way. By $(id\o \varphi)\Delta(a)\in A_t$ and $P(r\o 1)=(id\o \varepsilon_t)\Gamma(r)$, we have
    $$\Gamma(x)=\sum r_{(0)}\o \varepsilon_t(r_{(1)})\varphi(r_{(2)})=P(x\o 1),$$
so $x$ is a coinvariant. We need to show that $(r\lhd (r''\# \hat{a}), r')=(r, (r''\# \hat{a})\rhd r')$. It is equivalent to show that 
    $$\sum tr_{(0)}r'_{(0)}\varphi(r_{(1)}r'_{(1)})\left\langle S^{-1}(r_{(2)}), \hat{a} \right\rangle=\sum tr_{(0)}r'_{(0)}\varphi(r_{(1)}r'_{(1)})\left\langle r'_{(2)}, \hat{a} \right\rangle$$
for any $t\in R$. This is true because of the right invariance of $\varphi$. So the map $(\cdot, \cdot)$ is well defined. Next, for any $x\in R^{\text{co} A}$ we have
    $$(xr, r')=(id\o \varphi)\Gamma(xrr') =(x\o 1)(id\o \varphi)\Gamma(rr')=x(r, r').$$
In the second equality we use $\Gamma(xr)=(x\o 1)\Gamma(r)$. Similarly, by $\Gamma(rx)=\Gamma(r)(x\o 1)$ we can obtain $(r,r')x=(r, r'x)$. So $(\cdot, \cdot)$ is a bimodule map.
\end{proof}

Finally, we generalize \cite[Theorem 4.3]{van1999galois} to the case of  weak algebraic quantum groups with a two-sided faithful integral $\varphi$.

\begin{proposition}\label{morita context}
Let $(A, \Delta)$ be a weak algebraic quantum group with a two-sided faithful integral $\varphi$ and $(R,\Gamma, P)$ an $A$-comodule algebra. Then 
    $$\{R\#\hat{A}, R^{\text{co} A}, R, R, [\cdot,\cdot],(\cdot,\cdot) \}$$
is a Morita context.
\end{proposition}
\begin{proof}
We only need to show that $[r, r']\rhd r''=r(r',r'')$ and $(r,r')r''=r\lhd [r',r'']$. And we have
     $$[r,r']\rhd r''=\sum (rr'_{(0)}\# \varphi(r'_{(1)}\cdot))\rhd r''=\sum rr'_{(0)} \varphi(r'_{(1)}\cdot)\rhd r''=\sum r(r'r'')_{(0)}\varphi((r'r'')_{(1)})=r(r',r'').$$
Next, for any $t\in R$,
     \begin{align*}
         t(r\lhd [r',r'']) &=\sum tr\lhd (r'r''_{(0)}\# \varphi(r''_{(1)}\cdot))\\
         &= \sum tr_{(0)}r'_{(0)}r''_{(0)}\varphi(S^{-1}(r_{(1)}r'_{(1)}\varepsilon_t(r''_{(1)})))\\
         &= \sum tr_{(0)}r'_{(0)}r''\varphi(S^{-1}(r_{(1)}r'_{(1)}))\\
         &= \sum tr_{(0)}r'_{(0)}r''\varphi(r_{(1)}r'_{(1)})\\
         &= t(r,r')r''.
     \end{align*}
In the third equality we use $P(r\o 1)=(id\o \varepsilon_t)\Gamma(r)$. In the fourth equality we use $\varphi\circ S=\varphi$. Cancel $t$, then we obtain the equation. 
\end{proof}

Since $\varphi$ is a two-sided faithful integral, by Lemma \ref{faithfulintegral} we have $\hat{A}=\{\varphi(a\cdot)\mid \forall a\in A\}$. Then the map $\alpha: a\mapsto \varphi(a \cdot)$ from $A$ to $\widehat{A}$ is bijective. 

\begin{corollary}\label{galois}
    Let $(A, \Delta)$ be a weak algebraic quantum group with a two-sided faithful integral $\varphi$ and $(R,\Gamma, P)$ an $A$-comodule algebra. Then the following are equivalent:
    \begin{enumerate}[(1)]
        \item $\Gamma$ is a Galois coaction.
        \item $\nu$ is surjective.
        \item $[\cdot, \cdot]$ is surjective.
    \end{enumerate}
\end{corollary}
\begin{proof}
    Note that we have $[\cdot,\cdot]=(id \o \alpha)\circ \nu$. Then we only need to show that the map $[\cdot, \cdot]$ is injective. Assume that $[r,r']=0$, then $\sum rr'_{(0)}\varphi(r'_{(1)}a)=0$ for all $a\in A$. Apply any $\omega\in A'$, then $\sum \varphi(\omega(rr'_{(0)})r'_{(1)}a)=0$ for all $a\in A$. Since $\varphi$ is faithful, we get $(\omega\o id)((r\o 1)\Gamma(r'))=0$ for any $\omega$. So $(r\o 1)\Gamma(r')=0$. Recall that $\Gamma$ is injective, we get $r\o r'=0$.
\end{proof}

As mentioned in \cite{van2024duality}, many examples have a faithful integral, and there are known examples have a faithful set of integrals not a single faithful one. So it is expected that results above (Proposition \ref{morita context} and Corollary \ref{galois}), proven in the case where a single \textit{two-sided} faithful integrals exists, will also have a similar results in the general situation. This calls for a deeper investigation into the integral theory of WMHAs, which is a subject for subsequent research.

\section{Groupoid actions, gradings and Morita context}\label{groupoid-example}		

\subsection{Groupoid actions}
	
Let $G$ be any groupoid, $\mathrm{C}G$ be the standard examples of weak multiplier Hopf algebras (weak algebraic quantum groups), and the dual of $kG$ is denoted by $\mathrm{C}(G)$. Assume that $A$ is a nondegenerate algebra. Let $A$ be a left $\mathrm{C}G$-module algebra. First we generalize \cite[Proposition 4.1]{caenepeel2007galois} to our setting. 
	
\begin{proposition}
$\{\lambda_e\rhd 1\mid e\in G^0 \}$ is a set of central orthogonal idempotents in $M(A)$.
\end{proposition}
\begin{proof}
Recall that $\Delta(\lambda_e)=\lambda_e\o \lambda_e$,  then by equation \eqref{modalg}, we have $(\lambda_e \rhd 1)(\lambda_e\rhd 1)=(\lambda_e\rhd 1)$, so $\lambda_e\rhd 1$ is idempotent. By equation $\eqref{right-r}$, we have
$$
(\lambda_e\rhd 1)a=\lambda_e\rhd (1(S(\lambda_e)\rhd a))=\lambda_e\rhd a.
$$
By equation $\eqref{left-r}$, then we get $(\lambda_e\rhd 1)a=\lambda_e\rhd a=a(\lambda_e\rhd 1)$. 
If $e_1\neq e_2\in G^0$, then 
\begin{align*}
(\lambda_{e_1}\rhd 1)(\lambda_{e_2}\rhd 1)a&=(\lambda_{e_1}\rhd 1)(\lambda_{e_2}\rhd a)\\
&=\lambda_{e_1}\rhd (\lambda_{e_2}\rhd a)\\
&=(\lambda_{e_1}\lambda_{e_2})\rhd a\\
&=0.
\end{align*}
So these elements are orthogonal. Finally, for any $a\in A, p\in G$, we have
$$
(\lambda_p \rhd 1_A)a =(\varepsilon_t(\lambda_p)\rhd 1)a \stackrel{\eqref{xru}}{=} \varepsilon_t(\lambda_p)\rhd a= \lambda_{t(p)}\rhd a.
$$
and 
$$
a(\lambda_p \rhd 1)=a(\varepsilon_t(\lambda_p)\rhd 1)\stackrel{\eqref{yru}}{=} S^{-1}(\lambda_{t(p)})\rhd a=\lambda_{t(p)}\rhd a.
$$
\end{proof}

If $G$ is finite, the identity in $\mathrm{C}G$ is $\sum \lambda_e$ where the sum is taken over all the units. If $G$ is infinite, the same formula gives the identity in the multiplier algebra $M(\mathrm{C}G)$. So we have
$$\big(\sum_{e\in G^0}\lambda_e\big)\rhd 1_{M(A)}=1_{M(\mathrm{C}G)}\rhd 1_{M(A)}=1_{M(A)}.$$
Now the identity in $M(A)$ can be denoted by  $\sum_{e\in G^0} (\lambda_e\rhd 1_{M(A)})$. 
		  
For any $p\in G$, we define the set
$$A_p=\{\lambda_p\rhd a \mid a\in A\}.$$

\begin{proposition}
    As sets, we have $A_p=A_{\t(p)}$ and $A_{p^{-1}}=A_{\s(p)}$.
\end{proposition}
\begin{proof}
    By $p=pp^{-1}p$ and $\t(p)=pp^{-1}$, we get the first equation. Similarly, by $p^{-1}=p^{-1}pp^{-1}$ and $\s(p)=p^{-1}p$, we can prove the second one.
\end{proof}

\begin{proposition}
$A_p$ is a nondegenerate subalgebra of $A$. 
\end{proposition}	
\begin{proof}
Obviously, the multiplication is closed. Assume that $(\lambda_{\t(p)}\rhd a)(\lambda_p\rhd x)=0$ for all $a\in A$. Then we have $$a(\lambda_{\t(p)}\rhd 1)(\lambda_p\rhd x)=a(\lambda_{\t(p)}\lambda_p\rhd x)=a(\lambda_p\rhd x)=0$$
for all $a$. By the nondegeneracy of $A$, we get $\lambda_p\rhd x=0$.
\end{proof}

$\lambda_p\rhd 1_{M(A)}=\lambda_{\t(p)}\rhd 1_{M(A)}$ is the identity in multiplier algebra $M(A_p)$. 
Next, consider the following two maps
$$\alpha_p: A_{\s(p)}\rightarrow A_{\t(p)}, \lambda_{\s(p)}\rhd a\mapsto \lambda_p\rhd(\lambda_{\s(p)}\rhd a)=\lambda_p\rhd a$$
and
$$
\beta_p: A_{\s(p)}\rightarrow A_{\t(p)}, \lambda_{\s(p)}\rhd a\mapsto \lambda_{\t(p)}\rhd a.
$$
 It is easy to show that $\alpha_p$ and $\beta_p$ are isomorphisms, and then they give us two actions of $G$. Define 
 $$\theta_{\t(p)}: A_{\t(p)}\rightarrow A_{\t(p)}, \lambda_p\rhd a\mapsto \lambda_{\t(p)}\rhd a.$$
Then we get $\theta_{\t(p)}\circ \alpha_p=\beta_p\circ id_{\s(p)}$, $id_{\s(p)}$ is the identity map from $A_{\s(p)}$ to $A_{\s(p)}$. This means that $(G, A_e, \alpha_p)$ and $(G, A_e, \beta_p)$ are equivalent in the sense of \cite[Definition 2.5]{yamanouchi1993duality}.

Since $A$ is a unital $\mathrm{C}G$-module, we have $\bigcup_{e\in G^0}A_e=A$. Let $0\neq b=\lambda_{\s(p)}\rhd a\in A_{\s(q)}$, if $\alpha_q(b)\in A_{\s(p)}$, then we have $p\rhd (q\rhd b)=\lambda_p\rhd (\lambda_q\rhd a)=(\lambda_p\lambda_q)\rhd a$. This implies that $pq$ is defined, otherwise $b=0$. Then the action is true. So far, we get the following result. 

\begin{proposition}
    Let $A$ be a $\mathrm{C}G$-module algebra, then we have a true action of $G$ on $A$ with $(G, A_e, \alpha_p)$.
\end{proposition}

Conversely, we have following result.

\begin{proposition}
Any true action of $G$ on $A$ makes $A$ into a $\mathrm{C}G$-module algebra. 
\end{proposition}
\begin{proof}
Let $(G, A_e, \alpha_p)$ be a true action of $G$ on $A$. Consider now the action of $\mathrm{C}G$ over $A$ given by
$$\lambda_p\rhd a=\begin{cases}
    \alpha_p(a), & \text{if}\quad  a\in A_{\s(p)},\\
    0, & \text{otherwise}.
\end{cases}$$

Take any $p, q\in G$. If $a\notin A_{\s(q)}$, then $\lambda_p\rhd (\lambda_q\rhd a)=0$. And we have $(\lambda_p\lambda_q)\rhd a=0$ if $pq$ is not defined. If $pq$ is defined, then $\lambda_{pq}\rhd a=0$ because $a\notin A_{\s(pq)}$. 

If $a\in A_{\s(q)}$, then $\lambda_p\rhd (\lambda_q\rhd a)=\lambda_p\rhd \alpha_q(a)$. If $pq$ is defined, then $\alpha_q(a)\in A_{\t(q)}=A_{\s(p)}$. So $\lambda_p\rhd \alpha_q(a)=\alpha_p(\alpha_q(a))=\alpha_{pq}(a)=\lambda_{pq}\rhd a=(\lambda_p\lambda_q)\rhd a.$ If $pq$ is not defined, then $\alpha_q(a)\notin A_{\s(p)}$(otherwise, by the trueness of the action, we obtain that $pq$ is defined). So $\lambda_p\rhd \alpha_q(a)=0=(\lambda_p\lambda_q)\rhd a$.  

Note that the action $\rhd$ is unital because the groupoid action $(G, A_e, \alpha_p)$ is on $A$, not in $A$. So $A$ is a unital $\mathrm{C}G$-module. Finally, we need to show that $\lambda_p\rhd (ab)=(\lambda_p\rhd a)(\lambda_p\rhd b)$ for any $p\in G, a,b\in A$.

If $ab\notin A_{\s(p)}$, then we have $a\notin A_{\s(p)}$ or $b\notin A_{\s(p)}$. So $\lambda_p\rhd (ab)=0=(\lambda_p\rhd a)(\lambda_p\rhd b)$.

If $ab\in A_{\s(p)}$, then $\alpha_p(ab)$ is meaningful. By the definition of groupoid action, we must have $a,b\in A_{\s(p)}$ and $\alpha_p(ab)=\alpha_p(a)\alpha_p(b)$. Then $\lambda_p\rhd (ab)=\alpha_p(ab)=\alpha_p(a)\alpha_p(b)=(\lambda_p\rhd a)(\lambda_p\rhd b)$. 

Thus, $A$ is a $\mathrm{C}G$-module algebra with the action $\rhd$.
\end{proof}

Combining the above two properties, we obtain the following theorem. The finite group (groupoid) version of the following result is well known, we refer to \cite[Proposition 1.2]{cohen1984group},  \cite[Proposition 2.2]{paques2014duality} and \cite{caenepeel2007galois}.

\begin{theorem}
Let $A$ be a nondegenerate algebra and $G$ a groupoid. Any true action of $G$ on $A$ makes $A$ into a $\mathrm{C}G$-module algebra. Conversely, if $A$ is a $\mathrm{C}G$-module algebra, this induces a true action of $G$ on $A$.
\end{theorem}

Let $(G, \{A_e\}_{e\in G^0}, \{\alpha_p\}_{p\in G})$ be a true action of $G$ on $A$. By the action theory in \cite{zhou2017duality}, we can form the smash product $A\# \mathrm{C}G$ with the multiplication 
$$(a\# \lambda_p)(b\# \lambda_q)=\begin{cases}
    a\alpha_p(b)\o \lambda_{pq}, & \text{if}\,\, (p,q)\in G^2 \,\,\text{and}\,\, b\in A_{\s(p)},\\
    0, &\text{otherwise}.
\end{cases}$$
This is the skew groupoid ring and we will denote it by $A\ast G$.

\subsection{Groupoid gradings}

Let $G$ be a groupoid. A $G$-graded algebra $A$ is a nondegenerate algebra $A$ such that 
\begin{itemize}
    \item $A=\bigoplus_{p\in G} A_p$,
    \item $A_pA_q=\begin{cases}
        \subseteq A_{pq},   &\text{if} (p, q)\in G^2,\\
        0, & \text{otherwise}.
    \end{cases}$
\end{itemize}
For any $a\in A$, the unique decomposition is denoted by $\sum_{p\in I_a}a_p$, where $I_a$ is the finite index set.
The next proposition is an infinite groupoid version of \cite[Proposition 3.1]{caenepeel2007galois}.

\begin{proposition}
    Let $G$ be a groupoid. There is a one-to-one correspondence between $\mathrm{C}G$-comodule algebras and $G$-graded algebras.
\end{proposition}
\begin{proof}
    Let $A$ be a $G$-graded algebra. Define $\Gamma(a)=\sum_{p\in I_a}a_p\o \lambda_p\in A\o kG$. Note that $\Gamma(a_p)=a_p\o \lambda_p$ and $\Delta(\lambda_p)=\lambda_p\o \lambda_p$, then  we get
    $$(id\o \Delta)\Gamma(a)=\sum_{p\in I_a} a_p\o \lambda_p\o \lambda_p=(\Gamma\o id)\Gamma(a).$$
    For any $a, b\in A$, we have 
    $$ab=(\sum_{p\in I_a}a_p)(\sum_{q\in I_b}b_q)=\sum_{\substack{ p\in I_a, q\in I_b\\ (p,q)\in G^2}}a_pb_q$$
    with $a_pb_q\in A_{pq}$. So 
    $$\Gamma(a)\Gamma(b)=(\sum_{p\in I_a}a_p\o \lambda_p)(\sum_{q\in I_b}b_q\o \lambda_q)=\sum_{\substack{ p\in I_a, q\in I_b\\ (p,q)\in G^2}}a_pb_q\o \lambda_{pq}=\Gamma(ab).$$
    Let $P$ be the element in $M(A\o \mathrm{C}G)$ defined as $\sum_{e\in G^0} 1_e\o \lambda_e$ where $1_e$ is the identity in the multiplier algebra $M(A_e)$. Then we have 
    $$P(a\o \lambda_q)=\sum_{\substack{p\in I_a\\ \t(p)=\t(q)}}a_p\o \lambda_{q} \text{\quad and \quad} (a\o \lambda_q)P=\sum_{\substack{p\in I_a\\ \s(p)=\s(q)}}a_p\o \lambda_{q}$$
    for all $a\in A,p\in G$. We also have
    $$((id\o \varepsilon_t)\Gamma(a))(1\o \lambda_q)=\sum_{p\in I_a} a_p\o \lambda_{\t(p)}\lambda_q=\sum_{\substack{p\in I_a\\ \t(p)=\t(q)}} a_p\o \lambda_q.$$
    This means that $P(a\o 1)=(id\o \varepsilon_t)\Gamma(a)$. So $P$ is the canonical idempotent. Thus, $A$ is a $\mathrm{C}G$-comodule algebra. 

    Conversely, let $(A,\Gamma, P)$ be a $\mathrm{C}G$-comodule algebra, then $A$ can be viewed as a $\mathrm{C}(G)$-module algebra, and denote the action of $\mathrm{C}(G)$ on $A$ by $\rhd$. For any $p\in G$, let $A_p=\{\delta_p\rhd a \mid a\in A\}$ where $\delta_p$ is the function that is 1 in $p$ and 0 everywhere else. For any $p\neq q$, if $\delta_p\rhd a=\delta_q\rhd b$, then $\delta_p^2\rhd a=(\delta_p\delta_q)\rhd b=0$, so $\delta_p\rhd a=0$. This means that $A_p\cap A_q=0$ for any $p\neq q$. Since the module is unital, we have $a=\sum_{p\in I_a} \delta_p\rhd a_p$. So $A=\bigoplus_{p\in G} A_p$.
    Recall that we have 
    $$\Delta(\delta_g)(\delta_j\o \delta_k)=
    \begin{cases}
        \delta_j\o \delta_k, & when \quad jk=g,\\
        0, & otherwise.
    \end{cases}$$
    So we have 
$$\Delta(\delta_g)=\sum_{\s(h)=\s(g)} \delta_{gh^{-1}}\o \delta_h\in  M(A\o A).$$
    Then for any $x=\delta_p\rhd u\in A_p, y=\delta_q\rhd v\in A_q$, if $pq$ is defined, then
    $$\delta_{pq}\rhd (xy)=\mu(\Delta(\delta_{pq})(\delta_p\o \delta_q)\rhd (u\o v))=\mu((\delta_p\o \delta_q)\rhd (u\o v))=xy,$$
    it means that $xy\in A_{pq}$. If $pq$ is not defined, then $\delta_g\rhd (xy)=0$ for any $g$, so $xy=0$. So $A$ is graded.
\end{proof}

\subsection{Groupoid Morita context}

We will use the WMHA $\mathrm{C}(G)$ to illustrate the results on Morita context obtained in Section \ref{moritacontext}. The dual algebra of $\mathrm{C}G$ is denoted by $\widehat{\mathrm{C}G}$. Assume that $A$ is graded by $G$, then $A$ is a $\mathrm{C}(G)$-module algebra with the action $\delta_g\rhd a=a_g$, we may construct the smash product $A\# \mathrm{C}(G)$. The product is given by 

$$(a\# \delta_p)(b\# \delta_q)=\begin{cases}
    a(b_{pq^{-1}})\o \delta_q, \quad & if\quad \s(p)=\s(q);\\
    0, \quad & otherwise.
\end{cases}$$ 

Consider the formal sum $\sum x_e$ where $x_e$ is  any element of $A_e$ and $e\in G^0$. When the sum is infinite,  we denote it by $x=\cdots+x_{\t(p)}+\cdots+x_{\s(p)}+\cdots$. Then $x$ is a multiplier in $M(A)$ such that 
$$ax=\sum_{p\in I_a}a_px_{\s(p)},\quad  xa=\sum_{p\in I_a} x_{\t(p)}a_p.$$

\begin{proposition}
    The fixed point algebra $A^{\mathrm{C}(G)}=\{\sum x_e\mid x_e\in A_e\}$.
\end{proposition}
\begin{proof}
    Take any $b_e\in A_e$, we have 
    $$ab_e=\sum_{q\in I_a} a_qb_e=\sum_{\substack{q\in I_a\\\s(q)=e}}a_qb_e,$$
    then $\delta_p\rhd (ab_e)=a_pb_e=(\delta_p\rhd a)b_e$. So $b_e$ is a fixed point. Next we show that any fixed point is of this form. Take any $c_p\in A_p, p\notin G^0$. We have $\delta_g\rhd (ac_p)\in A_g$ and $(\delta_g\rhd a)c_p=a_gc_p\in A_{gp}$ for all $a\in A$. Since $p$ is not unit, $\delta_g\rhd (ac_p)\neq (\delta_g\rhd a)c_p.$
\end{proof}

Take any $a,b\in A, t\in A^{\mathrm{C}(G)}$, $A$ is a $(A\# \mathrm{C}(G), A^{\mathrm{C}(G)})$-bimodule with the module structures: 
$$(a\# \delta_g)\rhd b=ab_g,\quad b\lhd t=bt,$$
 and $A$ is a $(A^{\mathrm{C}(G)}, A\# \mathrm{C}(G))$-bimodule with the module structures:
$$b\lhd (a\# \delta_g)=(ab)_{g^{-1}}, \quad t\rhd b=tb.$$
We denote the two bimodules by $V$ and $W$, respectively. By \cite[Proposition 4.8]{van2024duality}, there is a faithful two-sided integral $\varphi$ on $\mathrm{C}G$ which is given by 
$$\varphi(\lambda_p)=\begin{cases}
    1, p\in G^0,\\
    0, p\notin G^0.
\end{cases}$$ 
And we have $\varphi\circ S=\varphi$. The isomorphism between $\widehat{\mathrm{C}G}$ and $\mathrm{C}(G)$ is given by $\varphi(\lambda_p\cdot)\mapsto \delta_{p^{-1}}$. As $A$ is graded by $G$, $A$ can also be viewed as a $\mathrm{C}G$-comodule algebra with the coaction $\Gamma(a)=\sum_{p\in I_a} a_p\o \lambda_p$. Now the two maps $[\cdot, \cdot]$ and  $(\cdot, \cdot)$ are given by $$[a, b]=\sum_{p\in I_a} ab_p\# \delta_{p^{-1}}, \quad (a, b)=\sum_{\substack{p\in I_{ab}\\ p\in G^0}}(ab)_p.$$

Now we obtain the following result, which generalizes Cohen and Montgomery's result (see \cite[Proposition 2.4]{cohen1984group}) on finite groups to arbitrary groupoids.
\begin{example}
    Let $A$ be graded by any groupoid $G$. Then $\{A\# \mathrm{C}(G), V,W,A^{\mathrm{C}(G)},[\cdot,\cdot],(\cdot,\cdot)\}$ forms a Morita context.
\end{example}

Finally we use the pairing $\left\langle \mathrm{C}(G), \mathrm{C}G\right\rangle$ to illustrate the Theorem \ref{bidual}.

\begin{example}\label{groupoiddual}
    Assume that $R$ is graded by groupoid $G$. With the notation as above, we have
    $$(R\# \mathrm{C}(G))\# \mathrm{C}G\cong R\o (\mathrm{C}(G)\# \mathrm{C}G).$$
    Take any $r_p\in R_p, \delta_g,\delta_q\in \mathrm{C}(G), \lambda_h\in \mathrm{C}G, r'\in R$. We have
    \begin{align*}
        (r_p\o \delta_g\# \lambda_h)\rhd_3 (r'\# \delta_q)
        &= r_pr'\bar{\o}((\delta_g\# \lambda_h)\rhd \delta_q)\\
        &=r_pr'\bar{\o}\delta_g(\lambda_h\rhd \delta_q)\\
        &=\begin{cases}
            r_pr'\bar{\o}\delta_g, \quad&\text{if}\quad  gh=q;\\
            0, &\text{otherwise.}
        \end{cases}
    \end{align*}
    On the other hand, we have 
    $$\Phi_1(r\# \delta_g)=\sum_{p\in I_r,\s(p)=\t(g)}r_p\o \delta_{pg}.$$
    When $\s(p)=\t(q)=\t(g)$, we have
    \begin{align*}
        ((r_p\# \delta_g)\# \lambda_h)\rhd_2 (r'\# \delta_q)
        &= r_pr'\bar{\o} \delta_{pg}(\lambda_h\rhd \delta_{pq})
        =\begin{cases}
            r_pr'\bar{\o} \delta_{pg}, \quad&\text{if}\quad  gh=q;\\
            0, &\text{otherwise.}
        \end{cases} 
    \end{align*}
    Now it is easy to obtain 
    $$((r_p\# \delta_g)\# \lambda_h)\rhd_2 (r'\# \delta_q)=r_p\o \delta_{pg}\# \lambda_{g^{-1}p^{-1}gh}\rhd_3 (r'\# \delta_q)$$
    and
    $$(r_p\o \delta_g\# \lambda_h)\rhd_3 (r'\# \delta_q)=(r_p\# \delta_{p^{-1}g})\# \lambda_{g^{-1}pgh}\rhd_2 (r'\# \delta_q).$$
    Then the isomorphism is given explicitly by the above two equations.
\end{example}

\section*{Acknowledgment}
We thank Alfons Van Daele for his constant support and guidance.

\normalem
\bibliographystyle{abbrv} 
\bibliography{reference}

% \bibitem{LL}G. Liu, F. Li, On strongly groupoid graded rings and the corresponding clifford theorem.
\end{document}